\documentclass[preprint]{elsarticle}
\usepackage{amssymb,amsthm}
\usepackage{tikz}
\usepackage{color}
\newtheorem{thm}{Theorem}[section]
\newtheorem{cor}[thm]{Corollary}
\newtheorem{lem}[thm]{Lemma}
\newtheorem{prop}[thm]{Proposition}
\newdefinition{defn}{Definition}
\newdefinition{rmk}{Remark}
\newdefinition{examp}{Example}
\journal{}

\begin{document}
\begin{frontmatter}
\title{On Levi extensions of nilpotent Lie algebras}%
\author[ur]{Pilar Benito}%
\ead{pilar.benito@unirioja.es}
\author[ur]{Daniel de-la-Concepci\'on}
\ead{daniel-de-la.concepcion@alum.unirioja.es}
%\author[ur]{J.M. P\'erez-Izquierdo}
%\ead{jm.perez@unirioja.es}
\address[ur]{Dpto. Matem\'aticas y Computaci\'on, Universidad de La Rioja, 26004, Logro\~no, Spain}%
\begin{abstract}
Levi's theorem decomposes any arbitrary Lie algebra over a field of characteristic zero, as a direct sum of a semisimple Lie algebra (named Levi factor) and its solvable radical. Given a solvable Lie algebra $R$, a semisimple Lie algebra $S$ is said to be a Levi extension of $R$ in case a Lie structure can be defined on the vector space $S\oplus R$. The assertion is equivalent to $\rho(S)\subseteq \mathrm{Der}(R)$, where $\mathrm{Der}(R)$ is the derivation algebra of $R$, for some representation $\rho$ of $S$ onto $R$. Our goal in this paper, is to present some general structure results on nilpotent Lie algebras admitting Levi extensions based on free nilpotent Lie algebras and modules of semisimple Lie algebras. In low nilpotent index a complete classification will be given. The results are based on linear algebra methods and leads to computational algorithms
\end{abstract}
\end{frontmatter}
% ----------------------------------------------------------------

% ----------------------------------------------------------------
\section{Introduction}
Given a finite-dimensional Lie algebra $L$ of characteristic zero with solvable radical $R$, Levi's theorem \cite[Chapter III, Section 9]{Ja62} asserts that there exists a semisimple subalgebra $S$ of $L$ such that $L=S\oplus R$. Note that if $[x,y]$ denotes the product on $L$:
\begin{itemize}
\item[$\bullet$] $[S,S]\subseteq S$ ($S$ is a subalgebra of $L$);
\item[$\bullet$] $[L,R]\subseteq Nil(L)\subseteq R$, $N=Nil(L)$ the nilradical of $L$ (see \cite[Chapter I, Section 7]{Ja62});
\item[$\bullet$] $R$ is an $S$-module via the adjoint representation of the subalgebra $S$ in $R$, $\rho=ad\,S\mid_R$:
$$
\rho:S\to \mathfrak{gl}(R), \rho(s)(a)=ad\,s(a)=[s,a]
$$(the assertion is also valid for the nilradical $N$.)

\end{itemize}
The algebra $L$ will be call \emph{faithful} if the adjoint representation $\rho$ is faithful. This condition is equivalent to the fact that  $L$ contains no nonzero  semisimple ideals.

We note that $\rho(S)$ is a subalgebra of the derivation Lie algebra $Der\, R\leq \mathfrak{gl}(R)$, (a linear map $\delta:R\to R$ is a derivation iff  $\delta(ab)=\delta(a)b+a\delta(b)$ where $ab$ denotes the product in $R$). Moreover, the restriction of each $\delta\in \rho(S)$ to the nilradical $N$, provides a subalgebra of the derivation algebra $Der\, N$.

In this way, representations of semisimple Lie algebras over solvable ones appear as a tool in the natural problem of determining all the finite-dimensional Lie algebras. In fact, the Lie algebras with abelian radical ($R^2=0$) are nothing else but \emph{split null extensions} $L=S\oplus V$ where $S$ is  a semisimple Lie algebra and $V$ an arbitrary S-module according to the definition of Lie module (see \cite[Chapter I, Section 5]{Ja62}). The multiplication on the extension algebra $L=S\oplus V$ is given, for $u,v\in V$ and $s\in S$, by:
\begin{itemize}
\item[$\circ$] $[u,v]=0$;
\item[$\circ$]$[s,v]=-[v,s]=\rho(s)(v)=s\cdot v$, where $\rho:S\to \mathfrak{gl}(V)$ denotes the representation of $S$ onto $V$;
\end{itemize}
and $S$ is viewed as a Lie subalgebra of $L$. We also note that $V$ is an abelian ideal. If $V$ is a trivial module, the center in the null extension algebra $L=S\oplus V$ is just $Z(L)=V$. This situation provides a trivial decomposition and leads the following definition:

\begin{defn}A Lie algebra is said to be \emph{directly decomposable} (\emph{indecomposable}) in case it can (cannot) be decomposed into a direct sum of ideals.
\end{defn}
In characteristic zero, any indecomposable Lie algebra is a faithful algebra; the converse is not true in general. On the other hand, any Lie algebra $L$ with radical $R$ can be decomposed into the direct sum (as ideals) of a  semisimple Lie algebra and a faithful Lie algebra with the same radical. In case $L$ be decomposable, $L$ splits into a direct sum of ideals which are indecomposable as Lie algebras, and the radical $R$ is just the sum of the radicals of the ideals in the decomposition.

In a more general setting, in order to get a Lie algebra by gluing a semisimple algebra $S$ and a solvable one $R$, the following conditions are needed:
\begin{itemize}
\item[$\circ$] A representation $\rho:S\to \mathfrak{gl}(R)$; then $\rho([x,y])=[\rho(x),\rho(y)]$ and so, for $x,y\in S,a\in R$ we have:
$$
\rho([x,y])(a)=\rho(x)(\rho(y)(a))-\rho(y)(\rho(x)(a)),
$$
where $[x,y]$ is the product in $S$. Denoting by $\rho(x)(a)=x\cdot a$, the previous equation can be rewritten as:
$$
[x,y]\cdot a=x\cdot(y\cdot a)-y\cdot(x\cdot a).
$$
\item[$\circ$] $\rho(S)\subseteq Der\, R$.
\end{itemize}

The previous conditions completely the general construction of Lie algebras by gluing solvable and semisimple Lie algebras as the following result shows:

\begin{lem}\label{lem11} Given a semisimple Lie algebra $S$ and a solvable one $R$, the vector space, $L=S\oplus R$ admits a Lie algebra structure with solvable radical $R$ and $S$ as subalgebra if and only if there exists a representation $\rho:S\to \mathfrak{gl}(R)$ such that  $\rho(S)\subseteq Der\, R$. 
\end{lem}

\begin{proof}
 The direct is clear from the previous introductory comments. Assume then there exists a representation $\rho$ such that $\rho(S)\subseteq Der\, R$ and denote $\rho(s)(a)=s\cdot a$. The product in the vector space $L=S\oplus R$ is given by the corresponding multiplications on $S$ and $R$ viewed as subalgebras of $L$ and, for $s\in S, a\in R$ $[sa]=-[sa]=s\cdot a$ provides the desired Lie algebra structure on $L$ (the Jacobi identity is a straightforward computation from the hypothesis). 
\end{proof}

This is the natural starting point in \cite{OnKh75} to examine the problem of classifying finite dimensional Lie algebras with a given radical. The algebra obtained by gluing $S$ and $R$ through the representation $\rho$ is denoted by $S\oplus_\rho R$. In \cite{OnKh75}, for an arbitrary solvable Lie algebra $R$, the authors take a Levi subalgebra $S_0$ of $Der\, R$ and consider the faithful Lie algebra $L_0=S_0\oplus_{id} R$. The algebra $L_0$ is called \emph{universal Lie algebra with radical $R$} and, up to isomorphisms, $L_0$ does not depend on the choice of $S_0$. In fact, from Propositions 3 and 4 in \cite{OnKh75}:

\begin{thm}\label{OnKh75}\emph{[Onishchick-Khakimdzhanov, 1975]} Any faithful Lie algebra with radical $R$ is isomorphic to a subalgebra of the universal algebra $L_0$ of the form $S\oplus_{id} R$ where $S$ is a semisimple subalgebra of $S_0$. Moreover, given $S_1,S_2$ semisimple subalgebras of $S_0$, the subalgebras $S_1\oplus_{id} R$ and $S_2\oplus_{id} R$ are isomorphic if and only if there exists an automorphism $\varphi$ or $R$ such that $\varphi S_1\varphi^{-1}=S_2$.
\end{thm}

So, the classification problem for Lie algebras with solvable radical $R$ is solved by knowing the automorphisms and Levi factors of the derivations of $R$. A different approach is given in \cite{Sn10} and \cite{Tu92}, where some necessary conditions on the radical structure of an indecomposable (non solvable) Lie algebra are displayed. The conditions follows from Levi's theorem and representation theory of Lie algebras. As an application, in \cite{Tu92} all the real Lie algebras of dimension $\leq 9$ that admit a nontrivial Levi decomposition are found. The classification involves only the simple real Lie algebras $\mathfrak{so}_3(k)$ and $\mathfrak{sl}_2(k)$ as Levi factors. The ideas in these papers lead to the next definition:

\begin{defn}\label{leviextension}Given a solvable Lie algebra $R$, a semisimple Lie algebra $S$ is said to be a \emph{Levi extension} of $R$ if there exists a representation $\rho:S\to \mathfrak{gl}(R)$ such that $\rho(S)\subseteq Der\, R$. Moreover, the Levi extension $S$ will be say \emph{indecomposable, faithful or trivial Levi extension} if the corresponding Lie algebra $S\oplus_{\rho} R$ is indecomposable, and $\rho$ is a faithful or a trivial representation. 
\end{defn}

Our goal in the present paper is to study the structure of nilpotent Lie algebras admitting a Levi extension. We will consider the case when the Lie algebra obtained by gluing a nilpotent Lie algebra and a Levi extension is faithful. A good explanation of the interest of the subject is given in \cite{Sn10} and references therein. The nilpotent condition is not very restrictive because of \cite[Theorem 2.2 ]{Tu92} (see Proposition \ref{propTurkowski} in this paper) and the fact that the classification of solvable Lie algebras can be reduced to the classification of nilpotent Lie algebras according to \cite{Mal50}.

The paper splits into three sections apart from the Introduction. In Section 2 some basic definitions and generic results on the radical structure of Lie algebras admitting Levi extensions are reviewed; following the ideas in \cite{OnKh75}, a general construction will be given by gluing characteristically nilpotent Lie algebras and Heisenberg Lie algebras. Free nilpotent Lie algebras are introduced in Section 3 as a main tool in the study of nilpotent Lie algebras that admit Levi extension(s). The former algebras let us characterize the last ones by using representation theory of semisimple Lie algebras. Section 4 is devoted to the classification of 3-step nilpotent  ($N^3=0$), also known in the literature as metabelian algebras, and 4-step nilpotent ($N^4=0$) Lie algebras. Our classification results provide a constructive way of getting a huge variety of examples of nilpotent algebras admitting any simple Lie algebra as Levi extension(s). Most part of the results in the paper are based on linear algebra methods and leads to computational algorithms included in \cite{BeCo12}.

Throughout the paper, all the vector spaces and algebraic objects are finite-dimensional and considered over a field $k$ of characteristic zero. Standard definitions, notations and terminology can be found in \cite{Hu72} and \cite{Ja62}.

%%%%%%%%%%%%%%%%%%
%%%SECTION 2
%%%%%%%%%%%%%%%%%%

\section{Basic structure and preliminary examples}

Let $N$ be a finite-dimensional nilpotent Lie algebra, $Der\, N$ and $Aut\, N$ its derivation
algebra and automorphism group respectively. We will denote the product in $N$ as $[a,b]$ and $[U,V]$ denotes the linear span of the set $\{[u,v]:u\in U, v\in V\}$. From the definition of nilpotent algebra, the \emph{lower central series} (l.c.s. for short)$$N,\  N^2=[N,N],\  N^3=[N,N^2],\  \dots,\  N^j=[N,N^{j+1}],\  \dots$$of $N$ ends at $0$ and the largest $t$ for which $N^t\neq0$ is called the \emph{nilpotency index of $N$} ($t$-nilindex for short); in this case, we will say that $N$ is a \emph{$t$-nilpotent Lie algebra}, so $N^{t+1}=0$ or $N$ is \emph{$(t+1)$-step nilpotent}. The \emph{type of $N$} is just the dimension of $N/N^2$. Along the paper, $Z(N)$ will denote the center of $N$; for nilpotent algebras the center is a nontrivial ideal and it is invariant under any derivation of $N$, i.e., the center is a \emph{characteristic ideal}, as terms $N^j$ in the l.c.s. 

A \emph{torus} on $N$ is a commutative subalgebra of $Der\, N$ which consists of semi-simple endomorphisms. Following \cite[Lemma 2.8]{Sa82}, over fields of characteristic zero, all maximal (for the
inclusion) tori on $N$ have the same dimension. This dimensional invariant is called the \emph{rank of $N$} and it is well known that $rank\, N\leq type\, N$.

From Levi's Theorem, any Lie algebra splits as the direct sum of its maximal solvable ideal, the \emph{solvable radical}, and a semisimple Lie algebra. Recall that nilpotent Lie algebras are a special subclass of solvable Lie algebras and that a solvable Lie algebra contains a unique maximal nilpotent ideal named \emph{nilradical}. As a consequence of Lemma \ref{lem11}, we get:

\begin{cor}\label{corlem11}A solvable Lie algebra of characteristic zero admits a nontrivial Levi extension if and only if its derivation Lie algebra is not solvable. \hfill $\square$
\end{cor}

Moreover, from complete reducibility of representations of semisimple Lie algebras, we easily arrive at Theorems 2.1 and 2.2 in \cite{Tu92}:

\begin{prop}\label{propTurkowski}\emph{[Turkowski, 1992]} Let $L=S\oplus R$ a Levi decomposition for the Lie algebra $L$ and denote by $\rho$ the adjoint representation of the Levi factor $S$ onto the solvable radical $R$. Then, $R=N\oplus T$ where $N$ is the nilradical of $L$ and $T$ is a trivial module. In particular, if $\rho$ does not provide trivial modules, the radical $R$ is nilpotent. Moreover, if $\rho$ is irreducible, $R^2=0$ and therefore $L$ is a split null extension of $S$ and $R$ via $\rho$.
\end{prop}
\begin{proof}The first assertion follows from $[L,R]\subseteq N$ and the fact that $N$ is a characteristic ideal and implies immediately the second one. For the third assertion use that $R^2$ is a characteristic ideal and the solvability of $R$.
\end{proof}

On the other hand, for nilpotent Lie algebras that admit Levi extension(s), in \cite[Theorem 2]{Sn10} we find these basic structural conditions:

  \begin{prop}\label{basiconLevi}\emph{[Snobl, 2010]} Let $N$ be a $(t+1)$-step nilpotent Lie algebra of characteristic zero, $t\geq 2$ and $S$ be a Levi extension of $N$ with related representation $\rho$. Assume that the Lie algebra $L=S\oplus_\rho N$ is directly indecomposable. Then, $\rho$ is a faithful representation and there exists a decomposition of $N$ into a direct sum of $S$-modules 
$$
N=m_1\oplus m_2\oplus \dots \oplus m_t
$$
which holds the following properties ($2\leq j \leq t$):
\begin{itemize}
\item [i)] $m_1$ is an $S$-module that generates $N$ as subalgebra; in particular, $\rho\mid{m_1}$ is a faithful representation;
\item [ii)] $N^j=m_j\oplus N^{j+1}$ with  $m_j\subseteq [m_1,m_{j-1}]$;
\item[iii)] a module $m_j$ decomposes into a sum of some subset of irreducible components of the tensor representation $m_1\otimes m_{j-1}$; moreover, if $m_{j-1}=kx$ is of dimension $1$, the irreducible components of $m_j$ are among those of $m_1$.
\end{itemize}
  \end{prop}

Previous proposition points out that, for nilpotent algebras that admit Levi extension(s), the quotient $N^j/N^{j+1}$ of two consecutive terms in the l.c.s. of $N$ splits into irreducible pieces inside the tensor product $\otimes^jN/N^2$. This will be the central structural feature on our resuls. The last part of this section is devoted to show several examples of nilpotent Lie algebras admitting or not Levi extension(s) following the ideas in \cite{OnKh75}. The examples are obtained from natural decompositions of linear maps spaces involving symplectic Lie algebras. In the examples, we will remark some module decompositions to enlighten the results in Section 4.

\begin{examp}\label{ex1}Following \cite[Chapter 4, Section 4.6,]{Di77}, a \emph{Heisenberg algebra} $(\mathfrak{h},[x,y])$ is a Lie algebra  such that its center is one dimensional and equal to $\mathfrak{h}^2$. For  such an algebra, if we denote by $Z(\mathfrak{h})=kz$, the kernel of the alternating bilinear form $b:\mathfrak{h}\times \mathfrak{h}\to k$ given by $[xy]=b(x,y)z$ is just $kz$, so there exists a basis $\{x_1,\dots,x_n,y_1,\dots,y_n, z\}$, called \emph{standard basis} of $\mathfrak{h}$, such that $[x_i,y_j]=-[y_jx_i]=\delta_{ij}z$ and the rest of brackets of basic elements are zero. We will denote by $\mathfrak{h}_{n}$ the Heisenberg algebra generated by $2n+1$ vectors. 

Now we investigate the structure of the derivation algebra $Der\, \mathfrak{h}_n$ (an alternative matrix description can be found in \cite{Be94} and in \cite[Section 2]{JiMeZh96}). Let $V$ be an arbitrary complement of $Z(\mathfrak{h}_n)=kz$ in $\mathfrak{h}_n$. Then $\mathfrak{h}_n=V\oplus kz$ is a $2$-graded decomposition, so we can introduce the natural grading on the set of endomorphisms of $\mathfrak{h}_n$:$$End\,(\mathfrak{h}_n)=End_{\bar{0}}\,(\mathfrak{h}_n)\oplus End_{\bar{1}}\,(\mathfrak{h}_n).$$

Let $\delta\in Der\, \mathfrak{h}_n$ and consider the homogeneous decomposition $$\delta=\delta_0+\delta_1.
$$Since the center is a characteristic ideal of dimension $1$, we have that $\delta_0(z)=\alpha z$, $\delta_1(z)=0$ and $\delta_i\in Der\, \mathfrak{h}_n$. But $\delta_0\in Der\, \mathfrak{h}_n$ if and only if $d_0\vline_{_V}-\frac{\alpha}{2}id_V\in \mathfrak{sp}_{2n}(V,b)$ and therefore,$$\delta_0=f+ \frac{\alpha}{2}\,\widehat{id}_V
$$where $f\,\vline_{_V}\in \mathfrak{sp}_{2n}(V), f(z)=0$, the map $\widehat{id}_V$ is the identity on $V$ and $\widehat{id}_V(z)=2z$. Then:$$Der\, \mathfrak{h}_n=\{\delta: \delta\,\vline_{_V} \in  \mathfrak{sp}_{2n}(V,b), \delta(z)=0\} \oplus k\cdot \widehat{id}_V \oplus  \{\delta: \delta(V)\subset kz, \delta(z)=0\}.
$$The algebra $Der\, \mathfrak{h}_n$ is a Lie subalgebra of $\mathfrak{gl}(\mathfrak{h}_n)=End\,(\mathfrak{h}_n)^-$ with product $[\delta_1,\delta_2]=\delta_1\delta_2-\delta_2\delta_1$. Note that the solvable radical of $Der\, \mathfrak{h}_n$ is just the $(2n+1)$-dimensional space
\begin{eqnarray}\label{rn}
\mathfrak{r}_n=k\cdot \widehat{id}_V \oplus  \{\delta: \delta(V)\subset kz, \delta(z)=0\}
\end{eqnarray}
and the (abelian) nilradical is $\mathfrak{n}=\{\delta: \delta(V)\subset kz, \delta(z)=0\}$. The sublagebra $$\mathfrak{sp}_{2n}(V,b)^\sharp=\{\delta: \delta\, \vline_{_V} \in  \mathfrak{sp}_{2n}(V,b), \delta(z)=0\}$$is a Levi factor of $Der\, \mathfrak{h}_n$ isomorphic to the simple symplectic Lie algebra $ \mathfrak{sp}_{2n}(V,b)$. Under the adjoint representation of  $Der\, \mathfrak{h}_n$, the map $\widehat{id}_V$ acts on $\mathfrak{n}$ as the identity and the $\mathfrak{sp}_{2n}(V,b)^\sharp$-module $\mathfrak{n}$ is isomorphic to the dual module $V^\star=V(\lambda_1)^*$ of the natural $\mathfrak{sp}_{2n}(V,b)$-module $V$. Both $V$ and $V^*$ are isomorphic modules of fundamental weight $\lambda_1$ as highest weight. Note that $\mathfrak{h}_n=V(\lambda_1)\oplus V(0)$ through the representation $\rho=id$ of $\mathfrak{sp}_{2n}(V,b)^\sharp$ on $\mathfrak{h}_n$.  $\square$
\end{examp}

The remarks in previous Example \ref{ex1} can be summarized in the following result:

\begin{prop}\label{heisenberglevi}Any faithful Lie algebra with radical $\mathfrak{h}_{n}$ is isomorphic to a Lie algebra of the form $S\oplus_{id}\mathfrak{h}_{n}$ where $S$ is a semisimple subalgebra of the symplectic Lie algebra $\mathfrak{sp}_{2n}(V,b)^\sharp=\{\delta\in \mathfrak{gl}(\mathfrak{h}_n): \delta\,\vline_{_V} \in  \mathfrak{sp}_{2n}(V,b), \delta(z)=0\}$. Moreover, the universal Lie algebra with radical $\mathfrak{h}_{n}$, $L_0(\mathfrak{h}_{n})=\mathfrak{sp}_{2n}(V,b)^\sharp\oplus_{id}\mathfrak{h}_{n}$ is indecomposable.
 \end{prop}
\begin{proof}The decomposition as $\mathfrak{sp}_{2n}(V,b)^\sharp$-modules of $L_0$ is $V(2\lambda_1)\oplus V(\lambda_1)\oplus V(0)$, so the nontrivial submodules are $V(2\lambda_1)=\mathfrak{sp}_{2n}(V,b)^\sharp$, $V(0)=kz$, $V(\lambda_1)=\mathfrak{h}_n$ and any direct sum of two modules of them. A proper ideal of $L_0$ cannot contains the $S$-module $V(2\lambda_1)$. So, the only possibilities for ideals are $Z(\mathfrak{h}_n)=kz$ and $\mathfrak{h}_n$, which proves the last assertion.
\end{proof}
 
However, not all nilpotent Lie algebras admit a Levi extension. This is the case of filiform Lie algebras of dimension $\geq 4$ and the so called characteristically nilpotent Lie algebras. 
\begin{examp}\label{ex2} From \cite{OnVi94}, a Lie algebra $\mathfrak{F}_n$ of dimension $n\geq 3$ and {\small $(n-1)$}-nilindex is called \emph{$n$-filiform}. Note that $\mathfrak{F}_3$ is just the Heisenberg $3$-dimensional $\mathfrak{h}_1$.

Filiform algebras are of \emph{type $2$} and the quotient of two consecutive terms in their l.c.s. is 1-dimensional. This last assertion is the special feature for which this class of algebras does not have nontrivial Levi extension(s) for dimension $\geq 4$. This fact was first noticed in \cite[Main Theorem 1]{An06} for indecomposable Lie algebras. The proof given in this work uses arguments on basis and structure constants. Here we will stablish that, except for $\mathfrak{F}_3$, a filiform Lie algebra cannot be the radical of a nontrivial Levi decomposition. The proof follows easily from the structure result iii) in  Proposition \ref{basiconLevi}. Note that according to next Proposition and its Corollary, a corrigenda of Theorem 1 in \cite{An06} must be done. \hfill $\square$
\end{examp}

\begin{prop}\label{filiformNO}The unique filiform Lie algebra that admits a nontrivial Levi extension is the Heisenberg algebra $\mathfrak{h}_1=span<x,y,z>$. For this algebra, $Der\, \mathfrak{h}_1=\mathfrak{sp}_2(k\cdot x\oplus k\cdot y,b)^\sharp\oplus \mathfrak{r}_1$ as described in Example \ref{ex1} and therefore, the split $3$-dimensional simple Lie algebra is its unique nontrivial Levi extension. Moreover, the universal Lie with radical $\mathfrak{h}_1$ is the $6$-dimensional Lie algebra $L_0(\mathfrak{h}_1)=\mathfrak{sp}_2(k\cdot x\oplus k\cdot y,b)^\sharp\oplus_{id}\mathfrak{h}_1$ with basis
$$x,\ y\ ,z\ , h=\left(
\begin{array}{cc}
 1& 0 \\
0 &  -1  
\end{array}
\right),\  e=\left(
\begin{array}{cc}
 0& 1 \\
0 &  0  
\end{array}
\right),\  f=\left(
\begin{array}{cc}
 0& 0 \\
1 &  0  
\end{array}
\right)
$$and nonzero products: $xy=z$, $hx=x$, $hy=-y$, $ey=x$, $fx=y$, $he=2e$, $hf=-2f$ and $ef=h$. 
\end{prop}
\begin{proof} Assume first $S$ is a nontrivial Levi extension of $\mathfrak{F}_n$. Following i) of Proposition \ref{basiconLevi}, $\mathfrak{F}_n$ is an algebra generated by a $2$-dimensional module $m_1$. Since $S$ provides a nontrivial extension, $m_1$ must be a nontrivial module. So $S$ acts faithfully on $m_1$ and $S\cong \mathfrak{sl}_2(k)$, the simple $3$-dimensional split algebra for which $m_1$ is the irreducible $2$-dimensional $V(\lambda_1)=V(1)$. In case $n\geq 4$, we have $\mathfrak{F}_n=m_1\oplus \Sigma_{j=2}^{n-1}kx_j$, where $m_j=kx_j$ are trivial modules which is no possible from iii) in Proposition \ref{basiconLevi}. Then $n=3$, $\mathfrak{F}_3=\mathfrak{h}_1$ and Proposition \ref{heisenberglevi} completes the proof.
\end{proof}
\begin{cor} A filiform Lie algebra $\mathfrak{F}_n$ with $n\geq 4$ cannot be the solvable radical of any nonsolvable Lie algebra with nontrivial Levi decomposition. In particular, the derivation Lie algebra of a filiform algebra of dimension $\geq 4$ is a solvable algebra. \hfill $\square$
\end{cor} 

\begin{examp}\label{ex3} Following \cite{DiLi57}, for a Lie algebra $\mathfrak{n}$ and $\mathfrak{D}=\mathrm{Der}\, \mathfrak{n}$, we define let recursively:$$
\mathfrak{n}^{[1]}=\mathfrak{D}(\mathfrak{n})=\{\Sigma\, \delta_i(x_i): x_i\in \mathfrak{n}, \delta_i\in \mathfrak{D}\}\quad \mathrm{and}\quad \mathfrak{n}^{[k+1]}=\mathfrak{D}(\mathfrak{n}^{[k]}).
$$The algebra $\mathfrak{n}$ is said to be \emph{characteristically nilpotent} (for short $\mathfrak{CN}$-algebra) if $\mathfrak{n}^{[k]}=0$ for some $k$. The first known example of a $\mathfrak{CN}$-algebra is the $8$-dimensional algebra $\mathfrak{Dl}_8$ described in terms of the basis $\{a_i: 1\leq i\leq 8\}$ and the products: $[a_i,a_j]=-[a_j,a_i]$ and $[a_i,a_j]=0$ for $i<j$; $[a_1,a_2]=-[a_3,a_4]=a_5$, $[a_1,a_3]=[a_2,a_4]=a_6$, $[a_1,a_4]=-[a_2,a_6]=-[a_3,a_5]=a_7$ and $[a_1,a_5]=-[a_2,a_3]=[a_4,a_6]=-a_8$.

By using \cite[Theorem 1 ]{LeTo59}, we get that the class of characteristically nilpotent algebras equals the class of Lie algebras whose derivation algebra is nilpotent. So, the derivation algebra of a $\mathfrak{CN}$-algebra does not contain semisimple subalgebras and therefore, the algebras in this class can not have nontrivial Levi extensions according to Corollary \ref{corlem11}. \hfill $\square$
\end{examp}

Proposition 6.5 in \cite{LeLu71} allows us to recover $\mathfrak{CN}$-algebras to get nilpotent Lie algebras admitting Levi extension(s): Let $\mathfrak{n}$ be a characteristically nilpotent algebra such that $\{x\in \mathfrak{n}:[x,\mathfrak{n}]\subseteq Z(\mathfrak{n})\}\subseteq \mathfrak{n}^2$, $z\in \mathfrak{n}$ a central element and $V_1 $ a $2$-dimensional vector space. Combining these ingredients, in \cite{LeLu71} the authors endowed the vector space $\mathcal{L}_1=\mathfrak{n}\oplus V_1$ with a structure of nilpotent Lie algebra for which its derivation algebra splits as $
\mathrm{Der}\ \mathcal{L}_1=\mathfrak{sl}_2(k)\oplus \mathfrak{N}$, $\mathfrak{N}$ being a nilpotent ideal. Hence the universal Lie algebra with radical $\mathcal{L}_1$ is given by $L_0(\mathcal{L}_1)=\mathfrak{sl}_2(k)\oplus_{id} \mathcal{L}_1$. The special trick in this construction is to define a Heisenberg algebra structure on $V_1\oplus kz$. This result can be generalized as follows:

\begin{prop} Let $\mathfrak{n}$ be a characteristically nilpotent Lie algebra such that the ideal $\{x\in \mathfrak{n}:[x,\mathfrak{n}]\subseteq Z(\mathfrak{n})\}$ is contained in the derived subalgebra $\mathfrak{n}^2$ and let $z_0\in \mathfrak{n}$ a fixed central element. For any $2n$-dimensional vector space $V_n=\mathrm{span}<x_i,y_i: 1\leq i\leq n>$, consider the direct sum $$\mathcal{L}_n(\mathfrak{n}, z_0)=\mathfrak{n}\oplus V_n$$with extended product $[x_i,y_i]=-[y_i,y_i]x_0$, $[V_n,\mathfrak{n}]=0$ and $\mathfrak{n}$ as subalgebra. Then, $\mathcal{L}_n(\mathfrak{n}, z_0)$ is a nilpotent Lie algebra with the same nilindex that $\mathfrak{n}$ and$$
\mathrm{Der}\ \mathcal{L}_n(\mathfrak{n},z_0)=\mathfrak{sp}_{2n}(V_n,b)^\star\oplus \mathfrak{N}
$$where $\mathfrak{sp}_{2n}(V_n,b)^\star=\{\delta\in \mathfrak{gl}(\mathcal{L}_n(\mathfrak{n}, z_0)): \delta\,\vline_{V_n} \in  \mathfrak{sp}_{2n}(V_n,b), \delta\,\vline_{\mathfrak{n}}=0\}$ and
\[
\begin{array}{cll}
\mathfrak{N}&=&\{\delta:\delta\,\vline_{{V_n}}=0, \delta\,\vline_{\mathfrak{n}} \in Der\, \mathfrak{n}, \delta(z_0)=0\}\ \oplus\\
&&\\[-2ex]
&&\{\delta: \delta(V_n)\subseteq \mathfrak{n}^2, \delta(\mathfrak{n})\subseteq V_n, \delta\,\vline_{\mathfrak{n}^2}=0, ad_\mathfrak{n}\, \delta(u)=-b(u,\delta(\cdot))z_0\  \forall_{u\in V_n}\},
\end{array}
\]
is a nilpotent ideal which consists of nilpotent linear operators.

In particular, any faithful Lie algebra with radical $\mathcal{L}_n(\mathfrak{n}, z_0)$ is isomorphic to a Lie algebra of the form $S\oplus_{id}\mathcal{L}_n(\mathfrak{n}, z_0)$ where $S$ is a semisimple subalgebra of the symplectic algebra $\mathfrak{sp}_{2n}(V_n,b)^*$.
\end{prop}
\begin{proof} Along the proof we write $\mathcal{L}_n$ instead of $\mathcal{L}_n(\mathfrak{n},z_0)$. It is immediate to check that the product defined above gives a Lie algebra structure on $\mathcal{L}_n$. Note that $Z(\mathcal{L}_n)=Z(\mathfrak{n})$ and $\mathcal{L}_n^2=\mathfrak{n}^2$. Then $\mathcal{L}^i=\mathfrak{n}^i$ for $i\geq 2$ and therefore $\mathcal{L}$ is nilpotent.

Consider now the nondegenerate alternating form: $b:V_n\times V_n \to k$ defined by $[u,v]=b(u,v)z_0$. Since $\mathcal{L}=V_n\oplus \mathfrak{n}$ is a $2$-graded algebra, any $d\in End(\mathcal{L}_n)$ can be decomposed as a sum of homogeneous components $\delta=\delta_0+\delta_1$, $\delta_i\in End_{\bar i}(\mathcal{L}_n)$. Assuming $\delta([x,y])=[\delta(x),y]+[x,\delta(y)]$, we get that $\delta:\mathcal{L}_n \to \mathcal{L}_n\in Der\, \mathcal{L}_n$ if and only if the following conditions holds:
\begin{enumerate}
\item[a)]for $u,v \in V_n$, $b(u,v)\delta_0(z_0)=(b(\delta_0(u),v)+b(u,\delta_0(v)))z_0$;
\item[b)]$\delta_1\,\vline_{\mathfrak{n}^2}=0$ and $\delta_0\,\vline_{\mathfrak{n}}\in Der\, \mathfrak{n}$;
\item[c)]for $u\in V_n, a\in \mathfrak{n}$, $[\delta_1(u),a]=-b(u,\delta_1(a))z_0$.
\end{enumerate}

Let $\delta$ be an arbitrary derivation of $\mathcal{L}_n$. Since $\mathfrak{n}$ is characteristically nilpotent, zero is the unique possible eigenvalue of $\delta$. Then, assertions $a)$ and $b)$ imply $\delta(z_0)=\delta_0(z_0)=\delta_1(z_0)=0$ and $\delta_0\,\vline_{V_n}\in\mathfrak{sp}_{2n}(V_n,b)$. Moreover, from $\{x\in \mathfrak{n}:[x,\mathfrak{n}]\subseteq Z(\mathfrak{n})\}\subseteq \mathfrak{n}^2$ and c) we have that $\delta_1(V_n)\subseteq \mathfrak{n}^2$.

From previous information we can decompose the derivation algebra of $\mathcal{L}$ into three subspaces:\[
\begin{array}{ll}
Der\, \mathcal{L}_n=&\{\delta\in End_{\bar 0}(\mathcal{L}_n): \delta\,\vline_{V_n} \in  \mathfrak{sp}_{2n}(V_n,b), \delta\,\vline_\mathfrak{n}=0\}\ \oplus\\
&\\[-2ex]
&\{\delta\in End_{\bar 0}(\mathcal{L}_n):\delta\,\vline_{V_n} =0, \delta\,\vline_\mathfrak{n} \in Der\, \mathfrak{n}, \delta(z_0)=0\}\  \oplus\\
&\\[-2ex]
&\{\delta\in End_{\bar 1}(\mathcal{L}_n): \delta(V_n)\subseteq \mathfrak{n}^2, \delta\,\vline_{\mathfrak{n}^2}=0, ad_\mathfrak{n}\, \delta(u)=-b(u,\delta(\cdot))z_0\  \forall_{u\in V_n}\}.
\end{array}
\]Now, the vector space $\mathfrak{N}$ given by the second and third terms in the direct sum decomposition of $Der\, \mathcal{L}_n$ is and ideal. Moreover, the derivations in $\mathfrak{N}$ are nilpotent maps: for $\delta=\delta_0+\delta_1 \in N$ and using that $\mathfrak{n}$ is characteristically nilpotent, we can get $\delta^r(\mathcal{L}_n)\subseteq \mathfrak{n}^2=\mathcal{L}_n^2$ for some $s\geq 2$; so, the assertion follows from the nilpotency of $\mathcal{L}_n$.
\end{proof}

\begin{examp}\label{ex4}The center of the Lie algebra $\mathfrak{Dl}_8$ is $Z(\mathfrak{Dl}_8)=k\cdot e_7+k\cdot e_8$. It can be easily checked that $\{x\in \mathfrak{Dl}_8:[x,\mathfrak{Dl}_8]\subseteq Z(\mathfrak{Dl}_8)\}\subseteq \mathfrak{Dl}_8^2$, so for $n\geq 1$ and $(0,0)\neq(\alpha,\beta)\in k\times k$ we get the series of nilpotent Lie algebras ($4$-step nilpotent, as $\mathfrak{Dl}_8$):$$\mathcal{L}_n(\mathfrak{Dl}_8,\alpha e_7+\beta e_8)=\mathfrak{Dl}_8\oplus V_n.$$The universal Lie algebra with (nil)radical $\mathcal{L}_n(\mathfrak{Dl}_8,\alpha e_7+\beta e_8)$ is$$L_0(\mathcal{L}_n(\mathfrak{Dl}_8,\alpha e_7+\beta e_8))=\mathfrak{sp}_{2n}(V_n,b)^\star\oplus_{id}\mathcal{L}_n(\mathfrak{Dl}_8,\alpha e_7+\beta e_8).$$Hence $\mathcal{L}_n(\mathfrak{Dl}_8,\alpha e_7+\beta e_8)$ allows nontrivial Levi extensions through symplectic Lie algebras and their semisimple subalgebras. The $\mathfrak{sp}_{2n}(V_n,b)^\star$-module decomposition of $\mathcal{L}_n(\mathfrak{Dl}_8,\alpha e_7+\beta e_8)$ is $V(\lambda_1)\oplus 8V(0)$. \hfill$\square$
\end{examp}

%%%%%%%%%%%%%%%%%%%%%
%%%%%% SECTION 3
%%%%%%%%%%%%%%%%%%%%%

\section{Universal free nilpotent algebras and Levi extensions}

In this section we will introduce the main tool to describe all the possible nilpotent Lie algebras that admit faithful Levi extension(s). Instead of considering the derivation algebra of a fixed nilpotent Lie algebra, we propose an alternative way starting from semisimple Lie algebras. The nilpotent Lie algebras we are looking for will appear as quotients of universal nilpotent Lie algebras generated by modules of semisimple Lie algebras.

Given a finite set $X=\{x_1,\ldots,x_d\}$ and the vector space $\mathfrak{m}$ with basis $X$, the free associative algebra generated by $X$ can be defined as:$$\mathfrak{F}_X=k1\oplus\sum_{j\geq 1}^\infty\otimes^j\mathfrak{m}.$$In \cite{Ja62} the free Lie algebra generated by $X$, and denoted by $\mathfrak{FL_X}$, is defined as the subalgebra of $(\mathfrak{F}_X^-,[ab]=ab-ba)$ generated by $X$. Then,$$
\mathfrak{FL}_X=\sum_{m\geq 1}\, \mathfrak{FL}_m
$$where $\mathfrak{FL}_m$ is the subspace generated by the linear combinations of homogeneous elements of the form $[x_{i_1},\dots, x_{i_m}]=[\dots[x_{i_1}x_{i_2}]\dots x_{i_m}]$, with $x_{i_s}\in X$. Let us denote by $\mathfrak{FL}^m$ the ideal $\sum_{j\geq m}\ \mathfrak{FL}_j$, and, for a fixed $t\geq 1$, let us consider the quotient Lie algebra:
\begin{eqnarray}\label{universal}
\mathcal{N}_{d,t}=\frac{\mathfrak{FL}_X}{\mathfrak{FL}_X^{t+1}}\ .
\end{eqnarray}
According to \cite[Proposition 4]{Sa70} and \cite[Proposition 1.4]{Ga73}, the algebra in (\ref{universal}) verifies the following universal property:
\begin{prop}The algebra $\mathcal{N}_{d,t}$ is a $t$-nilpotent and graded Lie algebra of type $d$ and any other $t$-nilpotent Lie algebra of type $d$ is an homomorphic image of $\mathcal{N}_{d,t}$. \hfill $\square$
\end{prop}

In fact, for a given t-nilpotent Lie algebra $\mathfrak{n}$ of type $d$ with (minimal) generating set $G=\{e_1,\dots,e_d\}$, the correspondence $x_i\mapsto e_i$ extends naturally to the surjective homomorphism of Lie algebras $\theta_G:\mathcal{N}_{d,t}\to \mathfrak{n}$:
\begin{eqnarray}\label{tetaG}
\theta_G([x_{i_1},\dots, x_{i_m}])=[e_{i_1},e_{i_2},\dots,e_{i_m}].
\end{eqnarray}
Thus, $\mathfrak{n}$ is viewed as the quotient Lie algebra $\mathcal{N}_{d,t}/Ker\, \theta_G$ with $\mathcal{N}_{d,t}^t\not\subseteq Ker\, \theta_G\subseteq \mathcal{N}_{d,t}^2$. From now on, we will refer to $\mathcal{N}_{d,t}$ as \emph{the free nilpotent Lie algebra of $t$-nilindex generated by $\mathfrak{m}=span<x_1,\dots,x_d>$} and to $\theta_G$ as the \emph{natural homomorphism} of $\mathcal{N}_{d,t}$ onto $\mathfrak{n}$ induced by $G$.

On the other hand, from \cite[Proposition 2]{Sa70}, any linear map $\delta:\mathfrak{m}\to \mathcal{N}_{d,t}$ can be extended uniquely to a derivation $\widehat{d}$ of $\mathcal{N}_{d,t}$ by defining:
\begin{eqnarray}\label{extension}\widehat{\delta}([x_{i_1},x_{i_2},\dots,x_{i_j},\dots x_{i_s}])=\sum_{j=1}^s[x_{i_1},x_{i_2},\dots,\delta(x_{i_j}),\dots x_{i_s}].
\end{eqnarray}The maps $\widehat{\delta}$, let us define the monomorphism of Lie algebras:
\begin{eqnarray}\label{delta}
\Delta:\mathfrak{gl}(\mathfrak{m})\to Der\, \mathcal{N}_{d,t}\ ,\ \delta\mapsto \widehat{\delta}.
\end{eqnarray}
Hence the derivation algebra of the free nilpotent Lie algebra generated by $\mathfrak{m}$ can be described as:$$
Der\, \mathcal{N}_{d,t} = \{\widehat{\delta}: \delta\in  \mathfrak{sl}(\mathfrak{m})\}\ \oplus k\cdot \widehat{id_\mathfrak{m}}\ \oplus\ \{\widehat{\delta}:\delta \in Hom(\mathfrak{m},\mathcal{N}(d,t)^2)\}.$$Note that $ \mathfrak{R}_{d,t}=k\cdot \widehat{id_\mathfrak{m}}\ \oplus\ \{\widehat{\delta}:\delta \in Hom(\mathfrak{m},\mathcal{N}(d,t)^2)\}$ is the solvable radical of the derivation algebra and $\mathfrak{N}_{d,t}=\{\hat{\delta}:\delta \in Hom(\mathfrak{m},\mathcal{N}(d,t)^2)\}$ is a nilpotent ideal (we also note that $[\widehat{id_\mathfrak{m}},\delta]=\delta$ for any $\delta\in \mathfrak{N}_{d,t}$). Moreover, $\mathfrak{sl}(\mathfrak{m})^\sharp=\{\widehat{\delta}: \delta\in  \mathfrak{sl}(\mathfrak{m})\}$ is a Levi subalgebra of $Der\, \mathcal{N}_{d,t}$ isomorphic to the special linear algebra $\mathfrak{sl}(\mathfrak{m})$, a simple Lie algebra of Cartan  type $A_{d-1}$). Then:
 
\begin{prop}\label{free nilpotent}Any faithful Lie algebra of characteristic zero with radical the free nilpotent algebra $\mathcal{N}_{d,t}$ generated by $\mathfrak{m}$ is isomorphic to a Lie algebra of the form $S\oplus_{id}\mathcal{N}_{d,t}$ where $S$ is a semisimple subalgebra of the special linear algebra $\mathfrak{sl}(\mathfrak{m})^\sharp=\{\delta\in Der\, \mathcal{N}_{d,t}: \delta\mid_\mathfrak{m} \in  \mathfrak{sl}(\mathfrak{m})\}$. \hfill $\square$
 \end{prop}
Now, using the ideas in \cite[Section 2]{Sa70}, we will stablish an alternative way of finding all the Levi extensions of $\mathcal{N}_{d,t}$ and its homomorphic images instead of looking for semisimple subalgebras of $\mathfrak{sl}(\mathfrak{m})$.

\begin{lem}\label{alternative way} Let $\mathcal{N}$ be a free nilpotent algebra of characteristic zero generated by $\mathfrak{m}$. Then, $\mathcal{N}$ admits a Levi extension $S$ if and only if the vector space $\mathfrak{m}$ admits a $S$-module structure.
  \end{lem}
  
\begin{proof}For a given representation $\rho$ such that $\rho(S)\subseteq Der\, \mathcal{N}$, since  $\mathcal{N}^2$ is a $\rho(S)$-module, by complete reducibility, we can decompose $\mathcal{N}=\mathcal{N}^2\oplus V$ for some $S$-module $V$; note that, since $V$ generates $\mathcal{N}$, for any faithful extension, $V$ is a faithful module. Now, $\mathfrak{m}$ and $V$ are equidimensional and therefore $\mathfrak{m}$ can be viewed as $S$-module. For the converse, let $\rho$ be a representation of $\mathfrak{m}$; the map $\Delta \circ \rho$, with $\Delta$ as in (\ref{delta}), is a representation of $\mathcal{N}$ such that $\Delta \circ \rho(S)\subseteq Der\, \mathcal{N}$ and therefore the result follows.
 \end{proof}

\begin{prop}\label{existenceequivalence}Let $\mathcal{N}_{d,t}$ be the free t-nilpotent algebra of characteristic zero generated by $\mathfrak{m}=span<x_1,\dots,x_d>$ and $\mathfrak{n}$ be an arbitrary  t-nilpotent Lie algebra of type d. Then, the following assertions are equivalent:
\begin{enumerate}
\item [a)] $S$ is a Levi extension of $\mathfrak{n}$;
\item [b)] $\mathcal{N}_{d,t}$ admits $S$ as a Levi extension and there exists a minimal generating set $G$ of $\mathfrak{n}$ for which the natural Lie algebra homomorphism $ \theta_G$ is a $S$-module homomorphism. 
\end{enumerate}
\end{prop}
 %%%%%%%%%
\begin{proof} Assume firstly $S$ be a Levi extension of $\mathfrak{n}$ attached to the representation $\rho$, so $\rho(s)\in Der\, \mathfrak{n}$ for $s\in S$. By complete reducibility, we can decompose $\mathfrak{n}=\mathfrak{n}^2\oplus V$ for some $S$-module $V$. Take as generator set of $\mathfrak{n}$ any basis $G=\{e_1,\dots,e_d\}$ of $V$. Consider the natural homomorphism $\theta_G$ from $\mathcal{N}_{d,t}$ onto $\mathfrak{n}$ given by $\theta_G(x_i)=e_i$. The one-one linear map $\theta_{\mathfrak{m}}:\mathfrak{m} \to V$ given by $\theta_{\mathfrak{m}}(a)=\theta_G(a)$, lets define the representation $\rho':S\to \mathfrak{gl}(\mathfrak{m})$ by means of $\rho'(s)=\theta_{\mathfrak{m}}^{-1}\rho(s)\mid_V\theta_{\mathfrak{m}}$. Now, the homomorphism $\Delta \circ \rho'$, where $\Delta$ is as in (\ref{delta}), provides a representation of $\mathcal{N}_{d,t}$ such that $\Delta \circ \rho'(S)\subset Der\, \mathcal{N}_{d,t}$. For this representation we have:\[
\begin{array}{ll}
\theta_G\circ \Delta(\rho'(s))([x_{i_1},x_{i_2},\dots, x_{i_k}])=&\\
\ \ \sum_{j=1}^k[\theta_G(x_{i_1}),\theta_G(x_{i_2}),\dots,\rho(s)(\theta_G(x_{i_j})),\dots, \theta_G(x_{i_k})]=&\\
\ \ \rho(s)([\theta_G(x_{i_1}),\theta_G(x_{i_2}),{\footnotesize \dots},\theta_G(x_{i_j}),{\footnotesize \dots}, \theta_G(x_{i_k})])=\rho(s)\circ \theta_G([x_{i_1},x_{i_2},{\footnotesize \dots}, x_{i_k}]).&\\
\end{array}\]So $\theta_G\circ \Delta(\rho'(s))=\rho(s)\circ \theta_G$ which proves $b)$. The converse is clear.
\end{proof}

From previous discussion, we can state our main result:

\begin{thm}\label{alternativewayII} Let $S$ be a semisimple Lie algebra, $\mathfrak{m}$ a $d$-dimensional $S$-module given by the representation $\rho$ and $\mathcal{N}_{d,t}$ the free $t$-nilpotent Lie algebra generated by $\mathfrak{m}$. The $t$-nilpotent Lie algebras of type $d$ admitting $S$ as Levi extension are of the form $\mathcal{N}_{d,t}/I$ where $\mathcal{N}_{d,t}$ and the ideal $I$ are $S$-modules through the representation $\Delta\circ\rho$, with $\Delta$ as in (\ref{delta}) and $\mathcal{N}_{d,t}^t\not\subseteq I \subseteq \mathcal{N}_{d,t}^2$. In this case, $\mathcal{N}_{d,t}/I$ is regarded as a quotient $S$-module. Moreover, $\mathcal{N}_{d,t}/I$ is a faithful module if and only if $\mathfrak{m}$ also is.
 \end{thm}

\begin{proof} The result follows from Lemma \ref{alternative way} and Proposition \ref{existenceequivalence}. For the last assertion, note that the representation $\rho'$ of $S$ in $\mathcal{N}_{d,t}/I$ is given by $\rho'(a+I)=\Delta\circ\rho(s)(a)+I$. Since $\mathfrak{m}$ generates $\mathcal{N}_{d,t}$, $\rho'$ is faithful representation if and only if $Ker\, \rho'=\{s\in S:\Delta\circ\rho(s)(a)\in I, \forall_{a\in \mathfrak{m}}\}=0$. But $I\cap \mathfrak{m}=0$ and $\mathfrak{m}$ is a $\Delta\circ\rho$-invariant set. Hence $Ker\, \rho'=ker \, \rho$.
\end{proof}
 \begin{defn}Given a free nilpotent Lie algebra $\mathcal{N}$ admitting the semisimple algebra $S$ as a Levi extension through the representation $\rho$, the ideals of $\mathcal{N}$ which are $S$-modules will be called \emph{$S$-ideals}.
\end{defn}
We remark that $S$-ideals of $\mathcal{N}$ are just the ideals of the Lie algebra $S\oplus_\rho \mathcal{N}$ contained in $\mathcal{N}$.

%%%%%%%%%%%%%%%%%%
%%%SECTION 4
%%%%%%%%%%%%%%%%%%

\section{Low nilindex and Levi extensions}

Previous section stablishes that nilpotent Lie algebras admitting Levi extension(s) can be characterized by modules of semisimple Lie algebras as generating sets of free nilpotent Lie algebras. In low nilindex, we can construct a huge variety of examples, as displayed in Table \ref{ltsadjuntos-so-sl-excep}, using Theorem \ref{alternativewayII}, multilinear algebra and simple Lie algebras and their irreducible modules

Models of free Lie algebras of nilindex $2$ and $3$ can be built by symmetric and skewsymmetric powers of vector spaces. The Lie algebra $\mathcal{N}_{d,2}$ has dimension {\small $\left(\begin{array}{c}d+1\\ 2 \end{array}\right)$} and following \cite[Section 2]{Ga73}, this algebra is isomorphic to:$$\mathcal{N}(\mathfrak{m},2)=\mathfrak{m}\oplus \bigwedge^2\mathfrak{m},$$where $\mathfrak{m}$ is an arbitrary $d$-dimensional vector space with nonzero skewsymmetric product given by $[x,y]=x\wedge y$, for $x,y\in \mathfrak{m}$. Similarly, the dimension of the Lie algebra $N_{d,3}$ is:{\small $$\left(\begin{array}{c}d+1\\ 2 \end{array}\right)+2\left(\begin{array}{c}d+1\\ 3 \end{array}\right)$$}and, up to isomorphisms, it can be realized as follows: the tensor product $\mathfrak{m}\otimes \bigwedge^2 \mathfrak{m}$ decomposes as the direct sum of the vector spaces $\mathfrak{t}$ and $\mathfrak{s}$ where$$\mathfrak{t}=\mathrm{span}<x\otimes(y\wedge z)+\ y\otimes(z\wedge x)+\ z\otimes(x\wedge y):x,\ y,\ z\in\mathfrak{m}>\cong\bigwedge^3\mathfrak{m}$$and$$
\mathfrak{s}=\mathrm{span}<2x\otimes (y\wedge z) +y\otimes (x\wedge z) +z\otimes (y\wedge x):x,\ y,\ z\in\mathfrak{m}>\cong\frac{\mathfrak{m}\otimes\bigwedge^2 \mathfrak{m}}{\mathfrak{t}}.$$In fact:\[
\begin{array}{ccc}
  x\otimes(y\wedge z)&=&  \overbrace{\frac{1}{3}(x\otimes(y\wedge z)+\ y\otimes(z\wedge x)+\ z\otimes(x\wedge y))}^{\mathfrak{t}} +\\
  %&&\\
  & & \overbrace{\frac{1}{3}(2x\otimes (y\wedge z) +y\otimes (x\wedge z) +z\otimes (y\wedge x)).}^{\mathfrak{s}}
  \end{array}
\]Then, the vector space$$\mathcal{N}(\mathfrak{m},3)=\mathfrak{m}\oplus \bigwedge^2 \mathfrak{m}\oplus \mathfrak{s}
$$with skewsymmetric product given, for $x,y,z\in \mathfrak{m}$, by $[x,y]=x\wedge y$ and$$[x,y\wedge z]=\frac{2}{3}x\otimes (y\wedge z) +\frac{1}{3}y\otimes (x\wedge z) +\frac{1}{3}z\otimes (y\wedge x)$$(all other products being zero) is the free $3$-nilpotent Lie algebra of type $d$. %$$

Now assume $\mathfrak{m}$ is a faithful $S$-module for some semisimple Lie algebra $S$. Without loss of generality, $S$ can be viewed as a Lie subalgebra of $\mathfrak{sl}(\mathfrak{m})$ and $\mathfrak{m}$ as the module given by the representation $\rho=id$. Following Theorem \ref{alternativewayII}, the free nilpotent algebras of nilindex $2$ and $3$ inherit the module structure of $\mathfrak{m}$ by means of the derivations displayed in (\ref{extension}). In this case, for each $\delta\in S$, the derivation extension of a linear map $\delta:\mathfrak{m}\to\mathfrak{m}$ on $\bigwedge^2 \mathfrak{m}$ and $\mathfrak{s}$ is given by:
\begin{enumerate}
\item[a)] $\widehat{d}(x)=\delta(x)$,
\item[b)]$\widehat{d}([x,y])=\widehat{\delta}(x\wedge y)=\delta(x)\wedge y+x\wedge \delta(y)$
\end{enumerate}and 
\begin{enumerate}
\item[c)] $\widehat{\delta}([x,y\wedge z])=\widehat{\delta}( x\otimes(y\wedge z))= \frac{2}{3}(\delta(x)\otimes(y\wedge z)+\ x\otimes(\delta(y)\wedge z)+$
\item[] $\ x\otimes(y\wedge \delta(z))) +\frac{1}{3}(\delta(y)\otimes (x\wedge z) +y\otimes (\delta(x)\wedge z) +y\otimes (x\wedge \delta(z))) +$
\item[] $\frac{1}{3}(\delta(z)\otimes (y\wedge x) +z\otimes (\delta(y)\wedge x) +z\otimes (y\wedge \delta(x))),$
\end{enumerate}
for $x,y,z\in\mathfrak{m}$. So $\bigwedge^2\mathfrak{m}$ and $\mathfrak{s}$ are the modules provided by the natural action of $S$ on $\otimes^n\mathfrak{m}$ with $n=2,3$:\[
\begin{array}{lll}\label{moducondition}
 \widehat{\delta}( x\otimes y)&=& \delta(x)\otimes y+x\otimes \delta(y)\\
 \widehat{\delta}(x\otimes y\otimes z)&=& \delta(x)\otimes y\otimes z+x\otimes \delta(y)\otimes z+x\otimes y\otimes \delta(z).
\end{array}
\]Then, for $2$ and $3$ nilindex algebras Theorem \ref{alternativewayII} can be reformulated as:

\begin{thm}\label{alternativewaylow}Let $S$ be a semisimple Lie subalgebra of $\mathfrak{sl}(\mathfrak{m})$ where $\mathfrak{m}$ is a vector space of dimension $d$ regarded as the natural module through the representation $\rho=id$. Then, the extended module structure on $\mathcal{N}(\mathfrak{m},2)=\mathfrak{m}\oplus \bigwedge^2\mathfrak{m}$ and $\mathcal{N}(\mathfrak{m},3)=\mathfrak{m}\oplus \bigwedge^2\mathfrak{m}\oplus \mathfrak{s}$ is given by considering $\bigwedge^2\mathfrak{m}$ and $\mathfrak{s}$ as $S$-submodules provided for the natural action of $S$ on $\mathfrak{m}\otimes\mathfrak{m}$ and $\mathfrak{m}\otimes\bigwedge^2\mathfrak{m}$ respectively. Moreover:
\begin{enumerate}
\item[a)] The $2$-nilpotent algebras (also known as metabelian) of type $dim\, \mathfrak{m}$ that admit $S$ as faithful Levi extension are of the form $\displaystyle{\frac{\mathcal{N}(\mathfrak{m},2)}{I}}$ where $I$ is an $S$-submodule of $\bigwedge^2\mathfrak{m}$ other than $\bigwedge^2\mathfrak{m}$.
\item[b)] The $3$-nilpotent algebras of type $dim\, \mathfrak{m}$ that admit $S$ as faithful Levi extension are of the form $\displaystyle{\frac{\mathcal{N}(\mathfrak{m},3)}{J+[J,\mathfrak{m}]}}$ where $J$ is an $S$-module of $\bigwedge^2\mathfrak{m}\oplus \mathfrak{s}$ such that $\mathfrak{s}\not\subseteq J+[J,\mathfrak{m}]$.
\end{enumerate}
\end{thm}
\begin{proof}The first part of the statement follows from previous discussion. Assetions a) and b) are obtained from Theorem \ref{alternativewayII} and the fact that $S$-ideals inside $\bigwedge^2\mathfrak{m}$ are just $S$-submodules for metabelian algebras. In the $3$-nilpotent case, $S$-ideals are of the form $J+[J,\mathfrak{m}]$ where $J$ is an $S$-submodule of $\bigwedge^2\mathfrak{m}$: given an arbitrary $S$-module $J$, the vector space $[J,\mathfrak{m}]$ spanned by the products $[a,b]$, $a\in J, b\in \mathfrak{m}$ is an $S$-module and  since $[J+[J,\mathfrak{m}],N(\mathfrak{m},3)]=[J,\mathfrak{m}]$, we get that $J+[J,\mathfrak{m}]$ is an ideal. 
\end{proof}

\begin{examp}\label{ex5} The nilradical of the Lie algebra $L_0(\mathfrak{h}_1)=\mathfrak{sp}_2(k)\oplus_{id}\mathfrak{h}_1$ in Proposition \ref{filiformNO} is the free nilpotent $\mathcal{N}_{2,2}$ regarded as module for $\mathfrak{sp}_2(k)$ as follows: $V=k\cdot x\oplus k\cdot y$
is the irreducible $2$-dimensional module $\mathfrak{m}=V(\lambda_1)=V(1)$ for the simple $3$-dimensional Lie algebra $\mathfrak{sp}_2(k)=\mathfrak{sl}_2(k)$. From Clebsh-Gordan formula, $\mathfrak{m}\otimes \mathfrak{m}=V(2)\oplus V(0)$ and $\bigwedge^2 \mathfrak{m}=V(0)=k\cdot x\wedge y$. By setting $z=x\wedge y$, the Lie algebra $L_0(\mathfrak{h}_1)$ appears as $\mathfrak{sp}_2(k)\oplus_{\rho_{V(\lambda_1)}}\mathcal{N}(\mathfrak{m},2)$.

The free $3$-nilpotent of type $2$ is given by$$\mathcal{N}(\mathfrak{m},3)=\mathfrak{m}\oplus \bigwedge^2 \mathfrak{m} \oplus \mathfrak{m}\otimes \bigwedge^2 \mathfrak{m},$$generated by the set $\{x,y,x\wedge y, x\otimes x\wedge y,y\otimes x\wedge y\}$ and with nonzero skewsymmetric products:$$[x,y]=x\wedge y, [x,x\wedge y]=x\otimes x\wedge y, [y,x\wedge y]=y\otimes x\wedge y.
$$The module $\mathfrak{m}=V(\lambda_1)$ for $\mathfrak{sl}_2(k)$ also gives the universal Lie algebra of radical $\mathcal{N}(\mathfrak{m},3)$, as $\mathfrak{sl}_2(k)\oplus_{\rho_{V(\lambda_1)}}\mathcal{N}(\mathfrak{m},3)$ with module decomposition $V(2)\oplus V(1)\oplus V(0)\oplus V(1)$. \hfill $\square$
\end{examp}

\begin{rmk}\label{nota}
In the same vein of previous Example \ref{ex5}, in Table \ref{ltsadjuntos-so-sl-excep} we use irreducible modules of fundamental weights $\lambda_1,\lambda_2$ and any simple Lie algebra $S$ to get a huge variety of series of free nilpotent Lie algebras that admit $S$ as Levi extension. According to Theorem \label{alternativewaylow}, from $S$-modules, we can get many others. The algebras $\mathcal{N}(\mathfrak{m},2)$ and $\mathcal{N}(\mathfrak{m},3)$ appears in Table \ref{ltsadjuntos-so-sl-excep} first row; they have no proper $S$-modules fulfilling the conditions in Theorem \ref{alternativewaylow} (part a)or b)). The Heisenberg algebra $\mathfrak{h}_2$ appears as the quotient $\displaystyle{\frac{\mathcal{N}_{4,2}}{V(\lambda_2)}}$ according to 7th row and $\mathfrak{h}_n$ as $\displaystyle{\frac{\mathcal{N}_{2n,2}}{V(\lambda_2)}}$ for $n\geq 3$ if we look at 9th row. In all these cases, the universal Lie algebra of radical $\mathfrak{h}_n$ is built from the symplectic $\mathfrak{sp}_{2n}(k)$, a simple Lie algebra of Cartan type $C_n$, and its irreducible module $V(\lambda_1)$ as main ingredients.

In the particular case of $\mathfrak{sl}_2(k)$-Levi extensions, we work with irreducible modules are $V(n\lambda_1)=V(n)$ with standard basis $\{a_0,\dots,a_n\}$ as given in \cite[Section 7.2]{Hu72}. Using the basis $h,e,f$ of $\mathfrak{sl}_2(k)$ described in Proposition \ref{filiformNO}, the representation $V(n)$ is given explicitly by the formulas:\[
\begin{array}{ll}\label{slstandar}
 h\cdot v_i=(n-2i)a_i &   for\  0\leq i\leq n \\
 x\cdot a_0=0\ and  \ x\cdot a_i=(n-(i-1))a_{i-1} & for\  0< i \leq n\\
  y\cdot a_n=0\ and \  y\cdot a_i=(i+1)a_{i+1} & for\  0\leq i< n
\end{array}
\]A general construction of $2$-nilpotent Lie algebras that admit $\mathfrak{sl}_2(k)$ as a faithful Levi extension can be found in \cite{Be94}. This work along with the results in this paper yields to the computational approach to $\mathfrak{sl}_2(k)$-Levi extensions in \cite{BeCo12} where several Sage implementations are displayed to compute $t$-nilpotent Lie algebras for $t\geq 3$. \hfill $\square$
\end{rmk}

A nilpotent Lie algebra obtained as a quotient of a free nilpotent by an homogeneous ideal is called \emph{quasi-cyclic nilpotent Lie algebra} (see \cite{Sa70} and references therein for a basic definition and alternative characterizations). Free nilpotent and metabelian Lie algebras are quasi-cyclic; this is not the case for nilindex $\geq 3$. For quasi-cyclic $3-$nilpotent Lie algebras, Theorem \ref{alternativewaylow} can be stablished as follows:

\begin{cor} Let $S$ be a semisimple Lie subalgebra of $\mathfrak{sl}(\mathfrak{m})$ where $\mathfrak{m}$ is a vector space of dimension $d$ regarded as the natural module for $S$ corresponding to the representation $\rho=id$. The $3$-nilpotent quasi-cyclic algebras of type $dim\, \mathfrak{m}$ that admit $S$ as faithful Levi extension are of the form $\displaystyle{\frac{N(\mathfrak{m},3)}{P\oplus \{[P,\mathfrak{m}]+Q\}}}$ where $P$ and $Q$ are $S$-modules of $\bigwedge^2\mathfrak{m}$ and $\mathfrak{s}$ respectively such that $\mathfrak{s}\not\subseteq [P,\mathfrak{m}]+Q$.
\end{cor}
\begin{proof}Any $3$-nilpotent quasi-cyclic algebra that admits $S$ as Levi extension is of the form $N(4,\mathfrak{m})/I$ where $I$ is an homogeneous $S$-ideal. Then, $I=J+[J,\mathfrak{m}]=(J+[J,\mathfrak{m}])\cap \bigwedge^2\mathfrak{m}\oplus (J\cap\mathfrak{s})+[J,\mathfrak{m}]$. 
The $S$-submodule $P=(J+[J,\mathfrak{m}])\cap \bigwedge^2\mathfrak{m}$ satisfies $[J,\mathfrak{m}]=[P,\mathfrak{m}]$ and taking $Q=J\cap\mathfrak{s}$ the result follows.
\end{proof} 

In $3-$nilindex and for non-quasi-cyclic algebras, the description given in Theorem \ref{alternativewaylow} is not as clear as in the quasi-cyclic case. Next Proposition provides some necessary conditions on the existence of non-quasi-cyclic Lie algebras admitting Levi extension(s).

\begin{prop}\label{cone} Let $\mathfrak{n}$ be a $3$-nilpotent non-quasi-cyclic Lie algebra generated by $\mathfrak{m}$ and $S$ a Levi extension. Assume $\mathfrak{n}=N(\mathfrak{m},3)/I$ for some non-homogeneous $S$-ideal inside $N(\mathfrak{m},3)^2$. Then, there exists a $S$-irreducible submodule $B$ of $I$ which has nonzero and isomorphic projections on $\bigwedge^2\mathfrak{m}$ and $\mathfrak{s}$. In particular, $S\neq \mathfrak{sl}(\mathfrak{m})$ and $N(\mathfrak{m},3)^3$ is not an $S$-irreducible module.
\end{prop}
\begin{proof}Write $N(\mathfrak{m},3)^2=\bigwedge^2\mathfrak{m}\oplus \mathfrak{s}$. Since $I$ is non-homogeneous, by complete reducibility there is an irreducible $S$-submodule $B$ of $I$ such that $B\cap \bigwedge^2\mathfrak{m}=0$ and $B\cap \mathfrak{s}=0$. So the projection maps from $B$ onto $\bigwedge^2\mathfrak{m}$ and $\mathfrak{s}$ are monomorphisms, therefore both $\bigwedge^2\mathfrak{m}$ and $\mathfrak{s}$ contain an isomorphic copy of $B$. The last assertion follows from \cite[Proposition 8]{Sa70}.
\end{proof}
From Table \ref{ltsadjuntos-so-sl-excep} and Proposition \ref{cone} we have that for $V(\lambda_1)$ and a classical simple Lie algebra $S$ (Cartan types $A,B,C,D$), the quotients $N(V(\lambda_1))/I$ provides only quasi-cyclic Lie algebras. In \cite[Section 4]{Sa70} a $6$-nilpotent non-quasi-cyclic Lie algebra $\mathcal{L}$ of type $4$ and dimension $38$ that admit $\mathfrak{sl}_2(k)$ as Levi extension is given. This algebra is a quotient of $\mathcal{N}_{4,6}$ with module structure inherit from $V(1)\oplus V(1)$. Using Sage, we have found a $3$-nilpotent non-quasi-cyclic algebra as a quotient of the free nilpotent $\mathcal{N}(\mathfrak{m},3)$ where $\mathfrak{m}$ is regarded as the $\mathfrak{sl}_2(k)$-irreducible $V(10)$.

\begin{examp} Let $\{v_0,\dots,v_{10},w_0,\dots,w_{18}, z_0,\dots,z_{6}, x_0,\dots,x_{14}\}$ be a basis of the nilpotent Lie algebra $\mathcal{N}_{10,18,6,4}$ with multiplication given in Table \ref{nocuasiciclica}. As nilpotent algebra  $\mathcal{N}_{10,18,6,4}$ is  generated by the minimal set $\{v_0,\dots,v_{10}\}$. This algebra admits a $\mathfrak{sl}_2(k)$-Levi extension by  considering $\{v_0,\dots,v_{10}\}$, $\{w_0,\dots,w_{18}\}$, $\{z_0,\dots,z_{6}\}$ and $\{x_0,\dots,x_{14}\}$ as standard basis (see Remark \ref{nota}) of the $\mathfrak{sl}_2(k)$-irreducible modules $V(10)$, $V(18)$, $V(6)$ and $V(14)$ respectively.
\end{examp}

\begin{table}[htb]
\caption{Module structure of free nilpotent of nilindex $2$ and $3$ admitting Levi extension $\mathcal{L}$}\label{ltsadjuntos-so-sl-excep}
\medskip
\hskip -0.5cm
\begin{minipage}{\linewidth}
%\begin{center}
\begin{tabular}{lllll}
\hline

& & & &\\[-2ex]
%{\scriptsize  Type}&
{\scriptsize $\mathcal{L}$} & {\scriptsize $\mathfrak{m}$} & {\scriptsize  $\bigwedge^2\mathfrak{m}$} &
{\scriptsize  $\bigwedge^3\mathfrak{m}$} & {\scriptsize  $\displaystyle \frac{\mathfrak{m}\otimes \bigwedge^2\mathfrak{m}}{\bigwedge^3\mathfrak{m}}$}\\
& & & &\\[-2ex]
\hline
&&&&\\[-2ex]

 {\scriptsize  $A_1$} & {\scriptsize  $V(\lambda_1)$} & {\scriptsize  $k$} &
{\scriptsize $0$} & {\scriptsize $V(\lambda_1)$}\\
  & & & &\\[-2ex]
  
 {\scriptsize  $A_2$} & {\scriptsize  $V(\lambda_1)$} & {\scriptsize  $V(\lambda_2)$} &
{\scriptsize $k$} & {\scriptsize $V(\lambda_1\tiny{+}\lambda_2) $}\\

& {\scriptsize  $V(\lambda_2)$} & {\scriptsize  $V(\lambda_1)$} &
{\scriptsize $k$} & {\scriptsize $V(\lambda_1\tiny{+}\lambda_2) $}\\
  & & & &\\[-2ex]

{\scriptsize  $A_n, (n\geq 3)$}

& {\scriptsize  $V(\lambda_1)$} & {\scriptsize  $V(\lambda_2)$} & {\scriptsize  $V(\lambda_3)$}
  & {\scriptsize $V(\lambda_1\tiny{+}\lambda_2)$}\\
  & & & &\\[-2ex]
& & & &\\[-2ex]

 {\scriptsize  $B_3$} & {\scriptsize  $V(\lambda_1)$} & {\scriptsize  $V(\lambda_2)$} &
{\scriptsize $V(2\lambda_3)$} & {\scriptsize $V(\lambda_1\tiny{+}\lambda_2)\oplus V(\lambda_1)$}\\
  & & & &\\[-2ex]

{\scriptsize  $B_n,(n\geq 4)$} & {\scriptsize  $V(\lambda_1)$} & {\scriptsize  $V(\lambda_2)$} &{\scriptsize $V(\lambda_3)$}& {\scriptsize $V(\lambda_1\tiny{+}\lambda_2)\oplus V(\lambda_1)$}\\
& & & &\\[-2ex]
& & & &\\[-2ex]

 {\scriptsize  $C_2$} & {\scriptsize  $V(\lambda_1)$} & {\scriptsize  $V(\lambda_2)\oplus k$} &
{\scriptsize $V(\lambda_1)$} & {\scriptsize $V(\lambda_1\tiny{+}\lambda_2)\oplus V(\lambda_1)$}\\

& {\scriptsize  $V(\lambda_2)$} & {\scriptsize  $V(2\lambda_1)$} &
{\scriptsize $V(2\lambda_1)$} & {\scriptsize $V(2\lambda_1\tiny{+}\lambda_2)\oplus V(\lambda_2)$}\\
  & & & &\\[-2ex]

{\scriptsize  $C_n,(n\ge 3)$} & {\scriptsize  $V(\lambda_1)$} & {\scriptsize  $V(\lambda_2)\oplus k$} &
{\scriptsize $V(\lambda_1)\oplus V(\lambda_3)$} & {\scriptsize $ V(\lambda_1\tiny{+}\lambda_2)\oplus V(\lambda_1)$}\\
& & & &\\[-2ex]
& & & &\\[-2ex]

 {\scriptsize  $D_4$} & {\scriptsize  $V(\lambda_1)$} & {\scriptsize  $V(\lambda_2)$} &
{\scriptsize $V(\lambda_3\tiny{+}\lambda_4)$} & {\scriptsize $V(\lambda_1\tiny{+}\lambda_2)\oplus V(\lambda_1)$}\\
  & & & &\\[-2ex]

{\scriptsize  $D_n, (n\geq 5)$} & {\scriptsize  $V(\lambda_1)$} & {\scriptsize  $V(\lambda_2)$} &
{\scriptsize $V(\lambda_3)$} & {\scriptsize $V(\lambda_1\tiny{+}\lambda_2)\oplus V(\lambda_1)$}\\
& & & &\\[-2ex]
& & & &\\[-2ex]

{\scriptsize  $G_2$} & {\scriptsize  $V(\lambda_1)$} & {\scriptsize  $V(\lambda_1)\oplus V(\lambda_2) $} &
{\scriptsize  $V(2\lambda_1)\oplus V(\lambda_1)\oplus k$} &
{\scriptsize  $V(\lambda_1\tiny{+}\lambda_2)\oplus V(2\lambda_1)\oplus$}\\

&  & &
 &
{\scriptsize  $\quad V(\lambda_1)\oplus V(\lambda_2)$}\\

 & {\scriptsize  $V(\lambda_2)$} & {\scriptsize  $V(3\lambda_1)\oplus V(\lambda_2) $} &
{\scriptsize  $V(4\lambda_1)\oplus V(3\lambda_1)\oplus V(2\lambda_1)$} &
{\scriptsize  $V(3\lambda_1\tiny{+}\lambda_2)\oplus V(2\lambda_1\tiny{+}\lambda_2)\oplus$}\\

&  & &
{\scriptsize  $\quad V(2\lambda_2)\oplus V(\lambda_2)\oplus k$} &{\scriptsize  $\quad V(\lambda_1\tiny{+}\lambda_2)\oplus V(2\lambda_2)\oplus V(\lambda_2)\oplus$}
\\

&  & &
 &{\scriptsize  $\quad V(3\lambda_1)\oplus V(2\lambda_1)$}
\\
& & & &\\[-2ex]
& & & &\\[-2ex]

 {\scriptsize  $F_4$} & {\scriptsize  $V(\lambda_1)$}& {\scriptsize  $V(\lambda_1)\oplus V(\lambda_2)$} &
{\scriptsize  $V(2\lambda_1)\oplus V(\lambda_2)$} &
{\scriptsize  $V(2\lambda_1)\oplus V(\lambda_1)\oplus V(\lambda_2)\oplus$}\\

&  &  &
{\scriptsize  $\quad V(2\lambda_3)\oplus V(2\lambda_4)\oplus k$} &
{\scriptsize  $ \quad V(2\lambda_4)\oplus V(\lambda_1\tiny{+} \lambda_2)\oplus$}\\

&  &  &
 &
{\scriptsize  $\quad V(\lambda_1\tiny{+}2\lambda_4)\oplus V(\lambda_3\tiny{+} \lambda_4)$}\\
& & & &\\[-2ex]
& & & &\\[-2ex]

{\scriptsize  $E_6$} & {\scriptsize  $V(\lambda_1)$} & {\scriptsize  $V(\lambda_3)$} &
{\scriptsize  $V(\lambda_4)$} &
{\scriptsize  $V(\lambda_1\tiny{+}\lambda_3)\oplus V(\lambda_1\tiny{+}\lambda_6)\oplus $} \\

&  &  &
 &
{\scriptsize  $\quad V(\lambda_2)$}\\
& & & &\\[-2ex]
& & & &\\[-2ex]

{\scriptsize  $E_7$} & {\scriptsize  $V(\lambda_1)$} & {\scriptsize  $V(\lambda_1)\oplus V(\lambda_3)$} &
{\scriptsize  $V(2\lambda_1)\oplus V(\lambda_3)\oplus$} &
{\scriptsize  $V(\lambda_1\tiny{+}\lambda_3)\oplus V(\lambda_1\tiny{+}\lambda_6)$} \\

& &  &
{\scriptsize  $\quad V(\lambda_4)\oplus V(\lambda_6)\oplus k$} &
{\scriptsize  $\quad V(\lambda_2\tiny{+}\lambda_7)\oplus V(2\lambda_1)\oplus$} \\

&  &  &
 &
{\scriptsize  $\quad 2V(\lambda_1)\oplus V(\lambda_3)\oplus V(\lambda_6)$}\\
& & & &\\[-2ex]
& & & &\\[-2ex]

{\scriptsize  $E_8$} & {\scriptsize  $ V(\lambda_1)$} & {\scriptsize  $V(\lambda_1\tiny{+}\lambda_8)\oplus V(\lambda_3)\oplus$} &
{\scriptsize  $V(2\lambda_1\tiny{+}\lambda_8)\oplus V(\lambda_3\tiny{+}\lambda_8)\oplus $} &$\qquad\star$
\footnote{$2V(\lambda_5)\oplus 2V(\lambda_1+2\lambda_8)\oplus 2V(2\lambda_1)\oplus 3V(\lambda_7+\lambda_8)\oplus 3V(\lambda_6)\oplus 4V(\lambda_2)\oplus 2V(2\lambda_8)\oplus V(\lambda_1+\lambda_3)\oplus V(\lambda_1+\lambda_6)\oplus V(\lambda_2+\lambda_7)\oplus V(2\lambda_1+\lambda_8)\oplus V(\lambda_3+\lambda_8)\oplus 2V(\lambda_1+\lambda_2)\oplus V(\lambda_6+\lambda_8)\oplus 3V(\lambda_1+\lambda_7)\oplus 3V(\lambda_2+\lambda_8)\oplus 3V(\lambda_3)\oplus 5V(\lambda_1+\lambda_8)\oplus 3V(\lambda_1)\oplus3V(\lambda_7)\oplus  2V(\lambda_8).$}

\\

&  & {\scriptsize  $\quad V(\lambda_7)\oplus V(\lambda_8)$} &
 {\scriptsize  $\quad V(\lambda_1\tiny{+}\lambda_2)\oplus V(\lambda_6\tiny{+}\lambda_8)\oplus$}&
\\

&  &  &
 {\scriptsize  $\quad 2V(\lambda_1\tiny{+}\lambda_7)\oplus 2V(\lambda_2\tiny{+}\lambda_8)\oplus$}&
\\

&  &  &
 {\scriptsize  $\quad 2V(\lambda_7\tiny{+}\lambda_8)\oplus 3V(\lambda_1\tiny{+}\lambda_8)\oplus$}&
\\

&  &  &
 {\scriptsize  $\quad 2V(\lambda_1)\oplus V(\lambda_2)\oplus 2V(\lambda_3)\oplus$}&
\\

&  &  &
 {\scriptsize  $\quad V(\lambda_4)\oplus V(\lambda_6)\oplus 3V(\lambda_7)\oplus$}&
\\

&  &  &
 {\scriptsize  $\quad 3V(\lambda_8)\oplus V(\lambda_8)$}&
\\

& & & &\\[-2ex]
\hline
\end{tabular}
%\end{center}
\end{minipage}
\end{table}

{\scriptsize 
\begin{table}[htdp]
\caption{Non quasi-cyclic $3$-nilpotent $\mathcal{N}_{10,18,6,4}$ multiplication}\label{nocuasiciclica}
\begin{center}
\begin{tabular}{ll}
{\footnotesize $[v_0, v_1]= w_0$}&{\footnotesize $[v_3, v_5]=  \frac{42}{221}*w_7 - \frac{54}{143}*x_5 - 7*z_1$}\\
{\footnotesize $[v_0, v_2]= \frac{1}{2}*w_1$}&{\footnotesize $[v_3, v_6]=  \frac{420}{2431}*w_8 - \frac{2}{11}*x_6 + \frac{21}{5}*z_2$}\\
{\footnotesize $[v_0, v_3]= \frac{4}{17}*w_2 + x_0$}&{\footnotesize $[v_3, v_7]=  \frac{288}{2431}*w_9 + \frac{2}{143}*x_7 + \frac{23}{5}*z_3$}\\
{\footnotesize $[v_0, v_4]= \frac{7}{68}*w_3 + \frac{1}{2}*x_1$}&{\footnotesize $[v_3, v_8]=  \frac{150}{2431}*w_{10} + \frac{15}{143}*x_8 + \frac{1}{5}*z_4$}\\
{\footnotesize $[v_0, v_5]= \frac{7}{170}*w_4 + \frac{3}{13}*x_2$}&{\footnotesize $[v_3, v_9]=  \frac{5}{221}*w_{11} + \frac{1}{11}*x_9 - 3*z_5$}\\
{\footnotesize $[v_0, v_6]= \frac{1}{68}*w_5 + \frac{5}{52}*x_3$}&{\footnotesize $[v_3, v_{10}]=  \frac{1}{221}*w_{12} + \frac{5}{143}*x_{10} + z_6$}\\
{\footnotesize $[v_0, v_7]= \frac{1}{221}*w_6 + \frac{5}{143}*x_4 + z_0$}&{\footnotesize $[v_4, v_5]=  \frac{294}{2431}*w_8 - \frac{42}{143}*x_6 - \frac{49}{5}*z_2$}\\
{\footnotesize $[v_0, v_8]= \frac{1}{884}*w_7 + \frac{3}{286}*x_5 + \frac{1}{2}*z_1$}&{\footnotesize $[v_4, v_6]=  \frac{441}{2431}*w_9 - \frac{49}{143}*x_7 - \frac{49}{10}*z_3
$}\\
{\footnotesize $[v_0, v_9]= \frac{1}{4862}*w_8 + \frac{1}{429}*x_6 + \frac{1}{5}*z_2$}&{\footnotesize $[v_4, v_7]=  \frac{420}{2431}*w_{10} - \frac{2}{11}*x_8 + \frac{21}{5}*z_4$}\\
{\footnotesize $[v_0, v_{10}]= \frac{1}{48620}*w_9 + \frac{1}{3432}*x_7 + \frac{1}{20}*z_3$}&{\footnotesize $[v_4, v_8]=  \frac{105}{884}*w_{11} + \frac{9}{286}*x_9 + 7*z_5
$}\\
{\footnotesize $[v_1, v_2]=  \frac{5}{17}*w_2 - 3*x_0$}&{\footnotesize $[v_4, v_9]=  \frac{25}{442}*w_{12} + \frac{20}{143}*x_{10} - 7*z_6$}\\
{\footnotesize $[v_1, v_3]=  \frac{5}{17}*w_3 - x_1$}&{\footnotesize $[v_4, v_{10}]=  \frac{1}{68}*w_{13} + \frac{5}{52}*x_{11}$}\\
{\footnotesize $[v_1, v_4]=  \frac{7}{34}*w_4 - \frac{2}{13}*x_2$}&{\footnotesize $[v_5, v_6]=  \frac{294}{2431}*w_{10} - \frac{42}{143}*x_8 - \frac{49}{5}*z_4$}\\
{\footnotesize $[v_1, v_5]=  \frac{2}{17}*w_5 + \frac{3}{26}*x_3$}&{\footnotesize $[v_5, v_7]=  \frac{42}{221}*w_{11} - \frac{54}{143}*x_9 - 7*z_5$}\\
{\footnotesize $[v_1, v_6]=  \frac{25}{442}*w_6 +\frac{20}{143}*x_4 - 7*z_0$}&{\footnotesize $[v_5, v_8]=  \frac{81}{442}*w_{12} - \frac{27}{143}*x_{10} + 21*z_6$}\\
{\footnotesize $[v_1, v_7]=  \frac{5}{221}*w_7 + \frac{1}{11}*x_5 - 3*z_1$}&{\footnotesize $[v_5, v_9]=  \frac{2}{17}*w_{13} + \frac{3}{26}*x_{11}$}\\
{\footnotesize $[v_1, v_8]=  \frac{35}{4862}*w_8 + \frac{6}{143}*x_6 - \frac{4}{5}*z_2$}&{\footnotesize $[v_5, v_{10}]=  \frac{7}{170}*w_{14} + \frac{3}{13}*x_{12}$}\\
{\footnotesize $[v_1, v_9]=  \frac{4}{2431}*w_9 + \frac{23}{1716}*x_7 + \frac{1}{10}*z_3$}&{\footnotesize $[v_6, v_7]=  \frac{30}{221}*w_{12} - \frac{54}{143}*x_{10} - 35*z_6$}\\
{\footnotesize $[v_1, v_{10}]=  \frac{1}{4862}*w_{10} + \frac{1}{429}*x_8 + \frac{1}{5}*z_4$}&{\footnotesize $[v_6, v_8]=  \frac{15}{68}*w_{13} - \frac{27}{52}*x_{11}$}\\
{\footnotesize $[v_2, v_3]=  \frac{3}{17}*w_4 - \frac{9}{13}*x_2$}&{\footnotesize $[v_6, v_9]=  \frac{7}{34}*w_{14} - \frac{2}{13}*x_{12}$}\\
{\footnotesize $[v_2, v_4]=  \frac{15}{68}*w_5 - \frac{27}{52}*x_3$}&{\footnotesize $[v_6, v_{10}]=  \frac{7}{68}*w_{15} + \frac{1}{2}*x_{13}$}\\
{\footnotesize $[v_2, v_5]=  \frac{81}{442}*w_6 - \frac{27}{143}*x_4 + 21*z_0$}&{\footnotesize $[v_7, v_8]=  \frac{3}{17}*w_{14} - \frac{9}{13}*x_{12}$}\\
{\footnotesize $[v_2, v_6]=  \frac{105}{884}*w_7 + \frac{9}{286}*x_5 + 7*z_1$}&{\footnotesize $[v_7, v_9]=  \frac{5}{17}*w_{15} - x_{13}$}\\
{\footnotesize $[v_2, v_7]=  \frac{150}{2431}*w_8 + \frac{15}{143}*x_6 + \frac{1}{5}*z_2$}&{\footnotesize $[v_7, v_{10}]=  \frac{4}{17}*w_{16} + x_{14}$}\\
{\footnotesize $[v_2, v_8]=  \frac{243}{9724}*w_9 + \frac{9}{104}*x_7 - \frac{33}{20}*z_3$}&{\footnotesize $[v_8, v_9]=  \frac{5}{17}*w_{16} - 3*x_{14}$}\\
{\footnotesize $[v_2, v_9]=  \frac{35}{4862}*w_{10} + \frac{6}{143}*x_8 - \frac{4}{5}*z_4$}&{\footnotesize $[v_8, v_{10}]=  \frac{1}{2}*w_{17}$}\\
{\footnotesize $[v_2, v_{10}]=  \frac{1}{884}*w_{11} + \frac{3}{286}*x_9 + \frac{1}{2}*z_5$}&{\footnotesize $[v_9, v_{10}]=  w_{18}$}\\
{\footnotesize $[v_3, v_4]=  \frac{30}{221}*w_6 - \frac{54}{143}*x_4 - 35*z_0$}&
\end{tabular}
\end{center}
\label{default}
\end{table}}{\scriptsize 
\begin{table}[htdp]
%\caption{Non quasi-cyclic $4$-nilpotent admitting $\mathfrak{sl}_2(k)$ Levi extension}\label{nocuasiciclica}
\begin{center}
\begin{tabular}{llll}
{\footnotesize $[v_0, w_7]= x_0$}&{\footnotesize $[v_4, w_6]= \frac{49}{26}*x_3$}&{\footnotesize $[v_8, w_1]= -\frac{12}{91}*x_2$}&{\footnotesize $[v_2, z_6]= -\frac{57}{11492}*x_7$}\\
{\footnotesize $[v_0, w_8]= \frac{11}{14}*x_1$}&{\footnotesize $[v_4, w_7]= -\frac{31}{13}*x_4$}&{\footnotesize $[v_8, w_2]= \frac{45}{364}*x_3$}&{\footnotesize $[v_3, z_0]= \frac{3762}{20111}*x_2$}\\
{\footnotesize $[v_0, w_9]= \frac{55}{91}*x_2$}&{\footnotesize $[v_4, w_8]= -\frac{55}{13}*x_5$}&{\footnotesize $[v_8, w_3]= \frac{3}{7}*x_4$}&{\footnotesize $[v_3, z_1]= \frac{2508}{20111}*x_3$}\\
{\footnotesize $[v_0, w_{10}]= \frac{165}{364}*x_3$}&{\footnotesize $[v_4, w_9]= -\frac{55}{13}*x_6$}&{\footnotesize $[v_8, w_4]= \frac{60}{91}*x_5$}&{\footnotesize $[v_3, z_2]= -\frac{570}{20111}*x_4$}\\
{\footnotesize $[v_0, w_{11}]= \frac{30}{91}*x_4$}&{\footnotesize $[v_4, w_{10}]= -\frac{77}{26}*x_7$}&{\footnotesize $[v_8, w_5]= \frac{9}{13}*x_6$}&{\footnotesize $[v_3, z_3]= -\frac{2280}{20111}*x_5$}\\
{\footnotesize $[v_0, w_{12}]= \frac{3}{13}*x_5$}&{\footnotesize $[v_4, w_{11}]= -x_8$}&{\footnotesize $[v_8, w_6]= \frac{21}{52}*x_7$}&{\footnotesize $[v_3, z_4]= -\frac{2090}{20111}*x_6$}\\
{\footnotesize $[v_0, w_{13}]= \frac{2}{13}*x_6$}&{\footnotesize $[v_4, w_{12}]= \frac{14}{13}*x_9$}&{\footnotesize $[v_8, w_7]= -\frac{30}{91}*x_8$}&{\footnotesize $[v_3, z_5]= -\frac{152}{2873}*x_7$}\\
{\footnotesize $[v_0, w_{14}]= \frac{5}{52}*x_7$}&{\footnotesize $[v_4, w_{13}]= \frac{35}{13}*x_{10}$}&{\footnotesize $[v_8, w_8]= -\frac{297}{182}*x_9$}&{\footnotesize $[v_3, z_6]= -\frac{38}{2873}*x_8$}\\
{\footnotesize $[v_0, w_{15}]= \frac{5}{91}*x_8$}&{\footnotesize $[v_4, w_{14}]= \frac{85}{26}*x_{11}$}&{\footnotesize $[v_8, w_9]= -\frac{330}{91}*x_{10}$}&{\footnotesize $[v_4, z_0]= \frac{627}{5746}*x_3$}\\
{\footnotesize $[v_0, w_{16}]= \frac{5}{182}*x_9$}&{\footnotesize $[v_4, w_{15}]= \frac{29}{13}*x_{12}$}&{\footnotesize $[v_8, w_{10}]= -\frac{2343}{364}*x_{11}$}&{\footnotesize $[v_4, z_1]= \frac{399}{2873}*x_4$}\\
{\footnotesize $[v_0, w_{17}]= \frac{1}{91}*x_{10}$}&{\footnotesize $[v_4, w_{16}]= -x_{13}$}&{\footnotesize $[v_8, w_{11}]= -\frac{927}{91}*x_{12}$}&{\footnotesize $[v_4, z_2]= \frac{285}{5746}*x_5$}\\
{\footnotesize $[v_0, w_{18}]= \frac{1}{364}*x_{11}$}&{\footnotesize $[v_4, w_{17}]= -7*x_{14}$}&{\footnotesize $[v_8, w_{12}]= -15*x_{13}$}&{\footnotesize $[v_4, z_3]= -\frac{190}{2873}*x_6$}\\
{\footnotesize $[v_1, w_6]= -7*x_0$}&{\footnotesize $[v_5, w_2]= -21*x_0$}&{\footnotesize $[v_8, w_{13}]= -21*x_{14}$}&{\footnotesize $[v_4, z_4]= -\frac{665}{5746}*x_7$}\\
\end{tabular}
\end{center}
%\label{default}
\end{table}%
}{\scriptsize 
\begin{table}[htdp]
%\caption{default}
\begin{center}
\begin{tabular}{llll}
{\footnotesize $[v_1, w_7]= -\frac{37}{7}*x_1$}&{\footnotesize $[v_5, w_3]= -9*x_1$}&{\footnotesize $[v_9, w_0]= -\frac{3}{91}*x_2$}&{\footnotesize $[v_4, z_5]= -\frac{19}{221}*x_8$}\\
{\footnotesize $[v_1, w_8]= -\frac{352}{91}*x_2$}&{\footnotesize $[v_5, w_4]= -\frac{15}{13}*x_2$}&{\footnotesize $[v_9, w_1]= -\frac{1}{14}*x_3$}&{\footnotesize $[v_4, z_6]= -\frac{171}{5746}*x_9$}\\
{\footnotesize $[v_1, w_9]= -\frac{495}{182}*x_3$}&{\footnotesize $[v_5, w_5]= \frac{42}{13}*x_3$}&{\footnotesize $[v_9, w_2]= -\frac{8}{91}*x_4$}&{\footnotesize $[v_5, z_0]= \frac{171}{2873}*x_4$}\\
{\footnotesize $[v_1, w_{10}]= -\frac{165}{91}*x_4$}&{\footnotesize $[v_5, w_6]= \frac{63}{13}*x_4$}&{\footnotesize $[v_9, w_3]= -\frac{5}{91}*x_5$}&{\footnotesize $[v_5, z_1]= \frac{342}{2873}*x_5$}\\
{\footnotesize $[v_1, w_{11}]= -\frac{102}{91}*x_5$}&{\footnotesize $[v_5, w_7]= \frac{57}{13}*x_5$}&{\footnotesize $[v_9, w_4]= \frac{5}{91}*x_6$}&{\footnotesize $[v_5, z_2]= \frac{285}{2873}*x_6$}\\
{\footnotesize $[v_1, w_{12}]= -\frac{8}{13}*x_6$}&{\footnotesize $[v_5, w_8]= \frac{33}{13}*x_6$}&{\footnotesize $[v_9, w_5]= \frac{7}{26}*x_7$}&{\footnotesize $[v_5, z_4]= -\frac{285}{2873}*x_8$}\\
{\footnotesize $[v_1, w_{13}]= -\frac{7}{26}*x_7$}&{\footnotesize $[v_5, w_{10}]= -\frac{3}{13}*x_8$}&{\footnotesize $[v_9, w_6]= \frac{8}{13}*x_8$}&{\footnotesize $[v_5, z_5]= -\frac{342}{2873}*x_9$}\\
{\footnotesize $[v_1, w_{14}]= -\frac{5}{91}*x_8$}&{\footnotesize $[v_5, w_{11}]= -\frac{57}{13}*x_9$}&{\footnotesize $[v_9, w_7]= \frac{102}{91}*x_9$}&{\footnotesize $[v_5, z_6]= -\frac{171}{2873}*x_{10}$}\\
{\footnotesize $[v_1, w_{15}]= \frac{5}{91}*x_9$}&{\footnotesize $[v_5, w_{12}]= -\frac{63}{13}*x_{10}$}&{\footnotesize $[v_9, w_8]= \frac{165}{91}*x_{10}$}&{\footnotesize $[v_6, z_0]= \frac{171}{5746}*x_5$}\\
{\footnotesize $[v_1, w_{16}]= \frac{8}{91}*x_{10}$}&{\footnotesize $[v_5, w_{13}]= -\frac{42}{13}*x_{11}$}&{\footnotesize $[v_9, w_9]= \frac{495}{182}*x_{11}$}&{\footnotesize $[v_6, z_1]= \frac{19}{221}*x_6$}\\
{\footnotesize $[v_1, w_{17}]= \frac{1}{14}*x_{11}$}&{\footnotesize $[v_5, w_{14}]= \frac{15}{13}*x_{12}$}&{\footnotesize $[v_9, w_{10}]= \frac{352}{91}*x_{12}$}&{\footnotesize $[v_6, z_2]= \frac{665}{5746}*x_7$}\\
{\footnotesize $[v_1, w_{18}]= \frac{3}{91}*x_{12}$}&{\footnotesize $[v_5, w_{15}]= 9*x_{13}$}&{\footnotesize $[v_9, w_{11}]= \frac{37}{7}*x_{13}$}&{\footnotesize $[v_6, z_3]= \frac{190}{2873}*x_8$}\\
{\footnotesize $[v_2, w_5]= 21*x_0$}&{\footnotesize $[v_5, w_{16}]= 21*x_{14}$}&{\footnotesize $[v_9, w_{12}]= 7*x_{14}$}&{\footnotesize $[v_6, z_4]= -\frac{285}{5746}*x_9$}\\
{\footnotesize $[v_2, w_6]= 15*x_1$}&{\footnotesize $[v_6, w_1]= 7*x_0$}&{\footnotesize $[v_{10}, w_0]= -\frac{1}{364}*x_3$}&{\footnotesize $[v_6, z_5]= -\frac{399}{2873}*x_{10}$}\\
{\footnotesize $[v_2, w_7]= \frac{927}{91}*x_2$}&{\footnotesize $[v_6, w_2]= x_1$}&{\footnotesize $[v_{10}, w_1]= -\frac{1}{91}*x_4$}&{\footnotesize $[v_6, z_6]= -\frac{627}{5746}*x_{11}$}\\
{\footnotesize $[v_2, w_8]= \frac{2343}{364}*x_3$}&{\footnotesize $[v_6, w_3]= -\frac{29}{13}*x_2$}&{\footnotesize $[v_{10}, w_2]= -\frac{5}{182}*x_5$}&{\footnotesize $[v_7, z_0]= \frac{38}{2873}*x_6$}\\
{\footnotesize $[v_2, w_9]= \frac{330}{91}*x_4$}&{\footnotesize $[v_6, w_4]= -\frac{85}{26}*x_3$}&{\footnotesize $[v_{10}, w_3]= -\frac{5}{91}*x_6$}&{\footnotesize $[v_7, z_1]= \frac{152}{2873}*x_7$}\\
{\footnotesize $[v_2, w_{10}]= \frac{297}{182}*x_5$}&{\footnotesize $[v_6, w_5]= -\frac{35}{13}*x_4$}&{\footnotesize $[v_{10}, w_4]= -\frac{5}{52}*x_7$}&{\footnotesize $[v_7, z_2]= \frac{2090}{20111}*x_8$}\\
{\footnotesize $[v_2, w_{11}]= \frac{30}{91}*x_6$}&{\footnotesize $[v_6, w_6]= -\frac{14}{13}*x_5$}&{\footnotesize $[v_{10}, w_5]= -\frac{2}{13}*x_8$}&{\footnotesize $[v_7, z_3]= \frac{2280}{20111}*x_9$}\\
{\footnotesize $[v_2, w_{12}]= -\frac{21}{52}*x_7$}&{\footnotesize $[v_6, w_7]= x_6$}&{\footnotesize $[v_{10}, w_6]= -\frac{3}{13}*x_9$}&{\footnotesize $[v_7, z_4]= \frac{570}{20111}*x_{10}$}\\
{\footnotesize $[v_2, w_{13}]= -\frac{9}{13}*x_8$}&{\footnotesize $[v_6, w_8]= \frac{77}{26}*x_7$}&{\footnotesize $[v_{10}, w_7]= -\frac{30}{91}*x_{10}$}&{\footnotesize $[v_7, z_5]= -\frac{2508}{20111}*x_{11}$}\\
{\footnotesize $[v_2, w_{14}]= -\frac{60}{91}*x_9$}&{\footnotesize $[v_6, w_9]= \frac{55}{13}*x_8$}&{\footnotesize $[v_{10}, w_8]= -\frac{165}{364}*x_{11}$}&{\footnotesize $[v_7, z_6]= -\frac{3762}{20111}*x_{12}$}\\
{\footnotesize $[v_2, w_{15}]= -\frac{3}{7}*x_{10}$}&{\footnotesize $[v_6, w_{10}]= \frac{55}{13}*x_9$}&{\footnotesize $[v_{10}, w_9]= -\frac{55}{91}*x_{12}$}&{\footnotesize $[v_8, z_0]= \frac{57}{11492}*x_7$}\\
{\footnotesize $[v_2, w_{16}]= -\frac{45}{364}*x_{11}$}&{\footnotesize $[v_6, w_{11}]= \frac{31}{13}*x_{10}$}&{\footnotesize $[v_{10}, w_{10}]= -\frac{11}{14}*x_{13}$}&{\footnotesize $[v_8, z_1]= \frac{1083}{40222}*x_8$}\\
{\footnotesize $[v_2, w_{17}]= \frac{12}{91}*x_{12}$}&{\footnotesize $[v_6, w_{12}]= -\frac{49}{26}*x_{11}$}&{\footnotesize $[v_{10}, w_{11}]= -x_{14}$}&{\footnotesize $[v_8, z_2]= \frac{855}{11492}*x_9$}\\
{\footnotesize $[v_2, w_{18}]= \frac{3}{14}*x_{13}$}&{\footnotesize $[v_6, w_{13}]= -\frac{119}{13}*x_{12}$}&{\footnotesize $[v_0, z_1]= -\frac{209}{442}*x_0$}&{\footnotesize $[v_8, z_3]= \frac{2565}{20111}*x_{10}$}\\
{\footnotesize $[v_3, w_4]= -35*x_0$}&{\footnotesize $[v_6, w_{14}]= -20*x_{13}$}&{\footnotesize $[v_0, z_2]= -\frac{1045}{6188}*x_1$}&{\footnotesize $[v_8, z_4]= \frac{9405}{80444}*x_{11}$}\\
{\footnotesize $[v_3, w_5]= -23*x_1$}&{\footnotesize $[v_6, w_{15}]= -35*x_{14}$}&{\footnotesize $[v_0, z_3]= -\frac{1045}{20111}*x_2$}&{\footnotesize $[v_8, z_5]= -\frac{1881}{40222}*x_{12}$}\\
{\footnotesize $[v_3, w_6]= -\frac{179}{13}*x_2$}&{\footnotesize $[v_7, w_0]= -x_0$}&{\footnotesize $[v_0, z_4]= -\frac{1045}{80444}*x_3$}&{\footnotesize $[v_8, z_6]= -\frac{1881}{6188}*x_{13}$}\\
{\footnotesize $[v_3, w_7]= -\frac{635}{91}*x_3$}&{\footnotesize $[v_7, w_1]= \frac{5}{7}*x_1$}&{\footnotesize $[v_0, z_5]= -\frac{95}{40222}*x_4$}&{\footnotesize $[v_9, z_0]= \frac{57}{40222}*x_8$}\\
{\footnotesize $[v_3, w_8]= -\frac{209}{91}*x_4$}&{\footnotesize $[v_7, w_2]= \frac{113}{91}*x_2$}&{\footnotesize $[v_0, z_6]= -\frac{19}{80444}*x_5$}&{\footnotesize $[v_9, z_1]= \frac{209}{20111}*x_9$}\\
{\footnotesize $[v_3, w_9]= \frac{55}{91}*x_5$}&{\footnotesize $[v_7, w_3]= \frac{83}{91}*x_3$}&{\footnotesize $[v_1, z_0]= \frac{209}{442}*x_0$}&{\footnotesize $[v_9, z_2]= \frac{95}{2366}*x_{10}$}\\
{\footnotesize $[v_3, w_{10}]= \frac{187}{91}*x_6$}&{\footnotesize $[v_7, w_4]= \frac{5}{91}*x_4$}&{\footnotesize $[v_1, z_1]= -\frac{209}{1547}*x_1$}&{\footnotesize $[v_9, z_3]= \frac{2090}{20111}*x_{11}$}\\
{\footnotesize $[v_3, w_{11}]= \frac{31}{13}*x_7$}&{\footnotesize $[v_7, w_5]= -x_5$}&{\footnotesize $[v_1, z_2]= -\frac{1045}{5746}*x_2$}&{\footnotesize $[v_9, z_4]= \frac{1045}{5746}*x_{12}$}\\
{\footnotesize $[v_3, w_{12}]= \frac{25}{13}*x_8$}&{\footnotesize $[v_7, w_6]= -\frac{25}{13}*x_6$}&{\footnotesize $[v_1, z_3]= -\frac{2090}{20111}*x_3$}&{\footnotesize $[v_9, z_5]= \frac{209}{1547}*x_{13}$}\\
{\footnotesize $[v_3, w_{13}]= x_9$}&{\footnotesize $[v_7, w_7]= -\frac{31}{13}*x_7$}&{\footnotesize $[v_1, z_4]= -\frac{95}{2366}*x_4$}&{\footnotesize $[v_9, z_6]= -\frac{209}{442}*x_{14}$}\\
{\footnotesize $[v_3, w_{14}]= -\frac{5}{91}*x_{10}$}&{\footnotesize $[v_7, w_8]= -\frac{187}{91}*x_8$}&{\footnotesize $[v_1, z_5]= -\frac{209}{20111}*x_5$}&{\footnotesize $[v_{10}, z_0]= \frac{19}{80444}*x_9$}\\
{\footnotesize $[v_3, w_{15}]= -\frac{83}{91}*x_{11}$}&{\footnotesize $[v_7, w_9]= -\frac{55}{91}*x_9$}&{\footnotesize $[v_1, z_6]= -\frac{57}{40222}*x_6$}&{\footnotesize $[v_{10}, z_1]= \frac{95}{40222}*x_{10}$}\\
{\footnotesize $[v_3, w_{16}]= -\frac{113}{91}*x_{12}$}&{\footnotesize $[v_7, w_{10}]= \frac{209}{91}*x_{10}$}&{\footnotesize $[v_2, z_0]= \frac{1881}{6188}*x_1$}&{\footnotesize $[v_{10}, z_2]= \frac{1045}{80444}*x_{11}$}\\
{\footnotesize $[v_3, w_{17}]= -\frac{5}{7}*x_{13}$}&{\footnotesize $[v_7, w_{11}]= \frac{635}{91}*x_{11}$}&{\footnotesize $[v_2, z_1]= \frac{1881}{40222}*x_2$}&{\footnotesize $[v_{10}, z_3]= \frac{1045}{20111}*x_{12}$}\\
{\footnotesize $[v_3, w_{18}]= x_{14}$}&{\footnotesize $[v_7, w_{12}]= \frac{179}{13}*x_{12}$}&{\footnotesize $[v_2, z_2]= -\frac{9405}{80444}*x_3$}&{\footnotesize $[v_{10}, z_4]= \frac{1045}{6188}*x_{13}$}\\
{\footnotesize $[v_4, w_3]= 35*x_0$}&{\footnotesize $[v_7, w_{13}]= 23*x_{13}$}&{\footnotesize $[v_2, z_3]= -\frac{2565}{20111}*x_4$}&{\footnotesize $[v_{10}, z_5]= \frac{209}{442}*x_{14}$}\\
{\footnotesize $[v_4, w_4]= 20*x_1$}&{\footnotesize $[v_7, w_{14}]= 35*x_{14}$}&{\footnotesize $[v_2, z_4]= -\frac{855}{11492}*x_5$}&{\footnotesize $$}\\
{\footnotesize $[v_4, w_5]= \frac{119}{13}*x_2$}&{\footnotesize $[v_8, w_0]= -\frac{3}{14}*x_1$}&{\footnotesize $[v_2, z_5]= -\frac{1083}{40222}*x_6$}&\\
\end{tabular}
\end{center}
%\label{default}
\end{table}%
}

\vfill
\newpage

\section*{Acknowledgements}
The authors would like to thank Spanish Government project  MTM 2010-18370-C04-03.

\end{document}